\documentclass[11pt]{amsart}
\usepackage[margin=1.25in]{geometry}
\usepackage{amsmath}
\usepackage{amssymb,amsfonts,latexsym, mathtools}

\usepackage{tikz-cd}
\usepackage{verbatim}
\usepackage[all]{xy}
\usepackage{xcolor}
\usepackage{comment}
\usepackage{mathrsfs}
\usepackage{url}
\usepackage{hyperref}
\usepackage{cleveref}

\newtheorem{thm}{Theorem}[section]
\newtheorem{cor}[thm]{Corollary}
\newtheorem{lem}[thm]{Lemma}
\newtheorem{prop}[thm]{Proposition}
\theoremstyle{definition}

\newtheorem{defn}[thm]{Definition}
\newtheorem{setup}[thm]{Setup}
\newtheorem{assumption}[thm]{Assumption}

\theoremstyle{remark}
\newtheorem{rem}[thm]{Remark}

\newtheorem{exam}[thm]{Example}

\numberwithin{equation}{section}

\usepackage{chngcntr}
\counterwithin{table}{section}

\newcommand\Pref[1]{{Proposition~\ref{#1}}}

\newcommand{\et}{{\operatorname{\acute{e}t}}}

\newcommand{\eps}{\varepsilon}

\newcommand{\mQ}{\mathbb{Q}}
\newcommand{\mC}{\mathbb{C}}
\newcommand{\Z}{\mathbb{Z}}

\newcommand{\mP}{\mathbb{P}}

\newcommand{\calS}{\mathcal{S}}
\newcommand{\calA}{\mathcal{A}}
\newcommand{\calU}{\mathcal{U}}
\newcommand{\calX}{\mathcal{X}}
\newcommand{\calP}{\mathcal{P}}
\newcommand{\kbar}{\bar{k}}

\DeclareMathOperator\Spec{Spec}
\DeclareMathOperator\Span{Span}
\DeclareMathOperator\Pic{Pic}

\DeclareMathOperator\Gal{Gal}

\newcommand\oline[1] {{\overline{#1}}}
\newcommand\uline[1] {{\underline{#1}}}

\newcommand\ind{{\operatorname{ind}}}

\newcommand\orb{{\operatorname{orb}}}

\newcommand\sub{{\,\subset\,}}
\renewcommand\Im{{\operatorname{Im}}}

\newcommand\GL[1][n] {{\operatorname{GL}_{#1}}}

\newcommand\gon{{\operatorname{gon}}}

\DeclareMathOperator\Mon{Mon}

\DeclareMathOperator\Sym{Sym}

\title{Hilbert Irreducibility for algebraic points}
\author{Borys Kadets}
\address{Borys Kadets, Einstein Institute of Mathematics\\ Hebrew University of Jerusalem\\
Jerusalem, Israel}
\email{kadets.math@gmail.com}
\urladdr{\url{http://bkadets.github.io}}

\author{Danny Neftin}
\address{Danny Neftin, Department of Mathematics\\ Technion -- Israel Institute of Technology\\ Haifa, Israel}
\email{dneftin@technion.ac.il}
\urladdr{\url{https://neftin.net.technion.ac.il/}}
\date{}

\begin{document} 
\begin{abstract}
      We study the following problem: given a covering of curves $\phi\colon X \to X_0$ over a number field $k$, and an integer $d$, when is the set \[\{p \in X_0(\overline{k})|\ \deg p = d, \text{ and the fiber } \phi^{-1}(p) \text{ is reducible over } k(p)\}\] finite? In case $X$ itself admits infinitely many degree $d$ points, we consider the modified problem where the images of degree $d$ points on $X$ are removed from the set. We prove a number of theorems ensuring a positive answer. As a consequence we show that for a fixed curve $X$ and all sufficiently high-degree indecomposable rational functions $\phi:X \to \mathbb{P}^1$ with $b$ branch points, the set of reducible fibers above degree $d<b/7-2$ points,  not containing a degree $d$ point from $X$, is finite.    
\end{abstract}
\maketitle
\section{Introduction} 

The classical Hilbert irreducibility theorem (\cite{hilbert1892ueber}, \cite{serre1989lectures}) asserts that given an irreducible polynomial $P(t,x) \in k[t,x]$ over a number field $k$, viewed as a family of polynomials in $x$ parameterized by $t$,  
the specialized polynomial $P(t_0, x)$ is irreducible for a density-one set of values $t_0 \in k$. In particular, given a polynomial $\phi(x)\in k[x]$, Hilbert's irreducibility implies that for most $\alpha \in k$ the polynomial $\phi(x)-\alpha$ is irreducible. Note that in this case there are also infinitely many $\alpha \in k$ for which $\phi(x)-\alpha$ is reducible: all $\alpha$ in the image $\phi(k)$ are such. Our goal is to investigate various versions of Hilbert's irreducibility for fibers of coverings of curves above algebraic points.

Let $k$ denote a finitely generated field of characteristic $0$ and $\oline k$ denote its algebraic closure. In this simpler setting of polynomial maps, our theorems imply the following:

\begin{thm}[\Cref{cor:hilb-for-poly}]\label{prop:intro}
    Suppose $\phi(x) \in k[x]$ is an indecomposable polynomial with $b$ branch points. Then for all but finitely many  $\alpha \in \overline{k}$ of degree $[k(\alpha):k] < b/6-1$, the polynomial $\phi(x)-\alpha$ is either irreducible or splits into a product of two irreducible factors one of which is linear. In particular, for such $\alpha$, the degree of an algebraic number $\beta \in \phi^{-1}(\alpha)$ over $k(\alpha)$ is either $\deg \phi,$ $\deg \phi -1,$ or $1$.
\end{thm}
\Cref{prop:intro} says that if a polynomial is indecomposable and sufficiently complicated (measured by the number of branch points), then Hilbert's irreducibility can be significantly strengthened: firstly, irreducibility of ``most'' fibers can be replaced by (almost) irreducibility of all but finitely many, and secondly, the fibers can be considered above all algebraic points of low degree, rather than points belonging to a specific number field. Extensions of the theorem of the first type have been much investigated, and it is the passage to low-degree points that is the principal novelty of this work. Since Hilbert's irreducibility holds for every fixed degree $d$ extension $K/k$, the results of this paper can be viewed as studying uniformity in the Hilbert irreducibility theorem for a varying field $K$.

In  geometric language, Hilbert's irreducibility asserts the irreducibility of many fibers of a rational function $\phi:X \to \mP^1_k$ above points in $\mP^1(k)$. For points of higher degree, one might as well consider coverings $\phi:X \to X_0$ with a general base curve $X_0$; for simplicity, we restrict to the case $X_0=\mP^1$ in the introduction, and treat   general curves $X_0$ afterwards.

A trivial, possibly infinite, source of reducible fibers of a map $\phi:X \to \mP^1$ is the value set $\phi(X(k))$. The \emph{finiteness problem} for Hilbert's irreducibility asks whether, under some conditions on $\phi$, one can prove that the fibers $\phi^{-1}(t_0)$ of $\phi$ are irreducible for all but finitely many $t_0 \in \mP^1(k)\setminus \phi(X(k))$. In the case of polynomial maps and $t_0\in \Z$, this problem is sometimes referred to as the Hilbert--Siegel problem (the name seems to originate from \cite{fried86}; see \cite[Section 5]{debes1999integral}, and \cite{konig2024reducible,BKN} for a more up-to-date account). For a decomposable map $\phi$, that is, when $\phi=\phi_1\circ\phi_2$ for maps $\phi_1,\phi_2$ of degree $>1$,  the fibers could be reducible over larger infinite sets  $\phi_1(k)$. Most of the literature, as suggested  in \cite[\S 4]{fried86}, concerns indecomposable maps $\phi$, and we shall restrict to such as well. 
Moreover, we first consider degree $\ell$ maps $\phi:X \to \mP^1$ of typical monodromy group $G$:   $G=A_\ell$ or $G=S_\ell$; other monodromy groups $G$ are discussed afterward.

For maps with few branch points, the finiteness problem might have a negative answer for fibers above low-degree points; see Examples \ref{exam:few-branch-points}, \ref{exam:weaker-vesionts}. Similarly, for us, the complexity of a covering is measured by the number $b$ of branch points rather than the degree. 

Our first main theorem roughly says that for high-degree indecomposable rational functions on a fixed curve $X$, finiteness in Hilbert's irreducibility holds for points of degree  at most $b/6$. 
We use $|X|_d$ to denote the set of degree $d$ points on a curve $X$. 

\begin{thm}[\Cref{cor:irred-rat-funct}]\label{thm:intro:S-HIT}
    Suppose $\phi:X \to \mP^1_k$ is a genus $g$ covering of  degree $\ell \geqslant \max\{2g,6\}$,   monodromy $A_\ell$ or $S_\ell$, and $b$  branch points. Then the fiber of $\phi$ is irreducible over all but finitely many points  in $\mP^1_k$ of degree $d< \min\{b/6-3,(\ell-2)/3\}$  outside $\phi(|X|_d)$.
 \end{thm}
 Note that on a fixed curve, indecomposable rational functions of high degree $\ell$ and at least $7$ branch points always have monodromy $A_\ell$ or $S_\ell$ by \cite[Theorem  1.1]{neftin2024monodromy}. 
The claim in the abstract then follows from Theorem \ref{thm:intro:S-HIT} since $b/7<(\ell-2)/3$ for a fixed genus $g$ and $\ell \gg 0$; see Remark \ref{b/7-abstract}. 

Often, by Hilbert's irreducibility theorem, the following (stronger) statement is meant. Given a covering $\phi: X \to \mP^1$ of monodromy group $G$, for most $p\in\mP^1(k)$ the image of the  Galois action on the preimages $\phi^{-1}(p)$ is also $G$; that is, $\Gal(k(\phi^{-1}(p))/k\cong G$. 
It is natural to consider the finiteness problem for the preservation of Galois groups; for this problem we obtain the following analogue of Theorem \ref{thm:intro:S-HIT}.

\begin{thm}[\Cref{cor:preserving-Galois-groups}]\label{thm:Galois-groups}     
    Let $\phi:X\to \mP^1_k$ be a genus $g$ covering of degree $\ell\geqslant 3g$, monodromy $A_\ell$ or $S_\ell$, and $b\geqslant 26$ branch points. 
    Then $\Gal(k(\phi^{-1}(P))/k(P))\cong A_\ell$ or $S_\ell$ for all but finitely many  points $P$ of degree $d\leqslant \min\{b/16,(\ell-4)/3\}$  outside $\phi(|X|_d)$. If, moreover, $\ell\geqslant 9$, then the Galois group of the fiber is $A_{\ell-1}$ or $S_{\ell-1}$  for all but finitely many points $P \in \phi(|X|_d)$ of degree $d\leqslant \min\{b/16, (\ell-4)/3\}$.
\end{thm}

The proof of the theorem applies bounds of Burness--Guralnick \cite{BG} on the index of primitive groups, which in turn rely on  the classification of finite simple groups. We also give a classification-free proof of Theorem \ref{thm:other-actions} when $\ell\geqslant 7900$.

We now turn to the case of general monodromy groups $G$. We will focus on seeking an analogue of \Cref{thm:Galois-groups}. In this case, it is natural for the results to be stated in terms of the parameters of a Galois $G$-covering $\tilde{X}\to X_0$. A few effects that already play a role in the $G=S_\ell$ case become more prominent for general groups.   First, the bounds of \Cref{thm:Galois-groups} involve the integer $\ell$, which will have to be replaced by a combinatorial parameter of the group $G$. Second, one cannot  distinguish between $A_\ell$ and $S_\ell$ fibers: indeed it is easy to give examples of $S_\ell$-covers for which infinitely many fibers over rational points have monodromy group $A_\ell$; see Example \ref{ex:Al-Sl}. For general groups, all normal subgroups $N\lhd G$ become suspect. With these two difficulties in mind, we state the following theorem, which is a consequence of the aforementioned group-theoretic results of Burness--Guralnick. We use $S_\ell \wr S_t$ to denote the {\it wreath product}, $S_\ell \wr S_t=S_\ell^t\rtimes S_t$, where $S_t$ acts by permuting the $t$ coordinates. 

\begin{thm}[\Cref{prop:monodromy}]\label{prop:intro-monodromy}
        Let $\tilde \phi:\tilde X\to X_0$ be a Galois covering of curves with monodromy group $G$ and $b$ branch points.  If there is an  infinite set $\mathcal S$ of points $P$ of degree $d<3b/28$ with $\Gal(k(\tilde\phi^{-1}(P))/k)\lneq G$, then one of the following holds: \\ 
    1) there is a normal subgroup  $1\neq N\lhd G$ for which the induced map $\tilde\psi:\tilde X/N\to X_0$ satisfies: there are  infinitely many points $P\in \mathcal S$ such that $\Gal(k(\tilde\psi^{-1}(P))/k)\lneq G/N$; \\
    2)  $A_\ell^t\leqslant G\leqslant S_\ell\wr S_t$ for some $t\geq1$, $1\leqslant k\leqslant \ell/2$, and $\Gal(k(\tilde \phi^{-1}(P))/k)$ is contained in $(S_k\times S_{\ell-k})\wr S_t$. 
\end{thm}

For example, consider the case of a Galois covering $\tilde\phi\colon \tilde X \to X_0$ with monodromy $G=\mathrm{PGL}_2(q)$. Assuming $q$ is odd, the only nontrivial normal subgroup of $G$ is $\mathrm{PSL}_2(q)$ of index $2$. In this case \Cref{prop:intro-monodromy} says that the fibers $\tilde\phi^{-1}(P)$ will have Galois group $\mathrm{PGL}_2(q)$ or $\mathrm{PSL}_2(q)$ above all but finitely many points $P$ of degree less than $(3/28)  \min\{b, b'\}$, where $b'$ is the number of branch points in the associated $\mathrm{PSL}_2(q)$-covering. 

In view of \Cref{prop:intro-monodromy}, the key case left to consider is that of groups of the so-called product type $A_\ell^t\leqslant G\leqslant S_\ell\wr S_t$.  The case $t=1$ is already taken care of in \Cref{thm:Galois-groups}. 
For $t\geqslant2$, consider the quotient map $\pi:G\to S_2\wr S_t$ modulo the minimal normal subgroup $A_\ell^t$ of $G$. The following theorem states that, under some assumptions, if infinitely many fibers of a map with monodromy $G$ have Galois group $H\neq G$, then $\pi(H)\neq \pi(G)$. This statement can then be applied inductively: for a given subgroup $H$, one applies the theorem with the base curve replaced with $\tilde{X}/(A_\ell^t\cdot H)$.

\begin{thm}[\Cref{cor:product-type-arithmetic}\label{thm:product-main} for $\eps=1/3$]
Suppose $\phi: X \to \mP^1_k$ is an indecomposable genus $g$ covering with $b$ branch points and monodromy group  $A_\ell^t\leqslant G\leqslant S_\ell\wr S_t$ acting on $t$-tuples from $\{1,\ldots,\ell\}$, for $\ell\geqslant5$ and $t\geqslant2$ (and so  $\deg \phi=\ell^t$). 
Let $\pi:S_\ell\wr S_t\to S_2\wr S_t$ be the natural projection modulo $A_\ell^t$. 
Then for all but finitely many  points $P$  of degree $d<\min\{b/12 -g/\ell^t,3b/56,\ell^t/(2t)\}$ 
such that $\pi(\Gal(k(\phi^{-1}(P))/k(P)))=\pi(G)$, one has  $\Gal(k(\phi^{-1}(P))/k(P))=G$.  
\end{thm}

Finally, we discuss some improved bounds, which can be obtained for coverings whose monodromy group $G$ is small compared to the number of branch points (and thus to the genus of the covering curve). The proofs rely on the work of Vojta \cite{vojta1992generalization} and Song--Tucker \cite{song2001arithmetic} on arithmetic discriminants.
\begin{thm}[\Cref{thm:large-branch-main}]\label{thm:large-branch}
    Suppose $\phi:X \to X_0$ is a covering of curves with monodromy group $G$ and $b$ branch points.  Let $n_{\max}$ denote the maximal index of a subgroup $H\leqslant G$, which itself is maximal among the subgroups with trivial core (in particular $n_{\max}\leqslant  |G|$). Then for all but finitely many degree $d<b/(2n_{\max})$ points $P$ of $X_0$, the  Galois group $\Gal(k(\phi^{-1}(P))/k(P))$  contains a nontrivial normal subgroup $1 \neq N \lhd G$. \end{thm}

For the problem of irreducibility of fibers of $S_\ell$-coverings, this approach gives the following.

\begin{thm}[\Cref{thm:vojta-new}]\label{intro:vojta}
    Suppose $\phi:X \to X_0$ is a covering of degree $\ell$, monodromy $A_\ell$ or $S_\ell$, and $b$ branch points. Then the fiber of $\phi$ is irreducible over all but finitely many points of degree $d<b/(2\ell).$    
\end{thm}

We note that the kinds of Hilbert's irreducibility theorems for a covering $\phi:X \to X_0$ that we are considering can only hold for points of degrees bounded in terms of $\phi$; in \Cref{appendix} we prove that any Galois type of a fiber can be achieved above infinitely many points of sufficiently high degree, even for covers of higher-dimensional varieties; see \Cref{main theorem}.

While all of the theorems above are of arithmetic nature, the proofs are geometric. 
The existence of degree $d$ points on curves $X$ is closely related to the structure of low-degree rational functions on $X$. Accordingly, our main results concern the \emph{gonality} of \emph{resolvent curves} $X_H=\tilde{X}/H$, that is, the minimal degree of rational function $X_H\to \mP^1_k$. In other words, given a $G$-covering $\tilde{X} \to \mP^1_k$,  
we show that the  curves $X_H$ cannot admit low-degree maps to $\mP^1_k$.

Many classical problems are connected to the finiteness problem for Hilbert's irreducibility; see the discussion following Remark 1.3 in \cite{neftin2024monodromy} for a large collection of such. Most of these admit natural extensions to higher-degree points, which have not been explored. 
As a specific example, for which  $f,g\in k[x]$ is the intersection $f\left(|\mP^1_k|_2\right) \cap g\left(|\mP^1_k|_2\right)$ of their value-sets on quadratic points finite?  We expect that the technique introduced in this paper will be useful for the study of this and many other questions.

\subsection{Sketch of the main argument}
Consider a degree $\ell$ rational function $\phi:\mP^1_k \to \mP^1_k$. Assume for simplicity that $\phi$ is ``generic'': the monodromy group of $\phi$ is $S_\ell$ and every ramified fiber of $\phi$ consists of $\ell-2$ unramified points and a single double point. We now describe how one can prove that  the fiber $\phi^{-1}(P)$ is either irreducible over $k(P)$, or contains a $k(P)$-rational point, for all but finitely many points $P\in \mP^1_k$ of small degree $d$.

Let $\tilde{X} \to \mP^1_k$ denote the Galois closure of $\phi$. Suppose the fiber $\phi^{-1}(P)$ over $P\in \mP^1_k$ is reducible and does not contain a rational point. Then the Galois group $\Gal(k(\phi^{-1}(P))/k(P)) \subset S_\ell$ is contained in $S_m \times S_{\ell-m}$ for some $m\geqslant2$, and so the curve $X_m\coloneq \tilde{X}/(S_m\times S_{\ell-m})$ has a $k(P)$-rational point above $P$.\footnote{This is a standard argument often used to reduce determining the exceptional set in Hilbert's irreducibility theorem to finding $X_m(k)$.}  We will show that the curves $X_m$ have only finitely many points of small degree, thus completing the proof.

One simple reason for a curve $X$ to have infinitely many points of degree $d$ is for $X$ to have a rational function $X \to \mP^1_k$ of degree $d$: the classical Hilbert irreducibility theorem then shows that $X$ admits infinitely many degree $d$ points. It turns out, as was discovered by Abramovich--Harris \cite{abramovich1991abelian} and Frey \cite{frey1994curves}, that a partial converse holds: if a curve admits an infinite number of degree $d$ points, then it admits a rational function of degree at most $2d$ (we will use a more refined version of this statement; see \Cref{prop:arithm-gonality}). Thus all we need to do is to show that the curves $X_m$ do not admit low-degree maps to $\mP^1_k$. We emphasize here that this problem is  geometric, rather than topological, which is a key difference between the cases $d=1$ and $d>1$: in the case $d=1$, Faltings' theorem  reduces the problem to showing that the genus of $X_m$ is bigger than $1$, which completely transforms the problem into a topological (or combinatorial) one. 

A direct approach to bounding the \emph{gonality} $\gon (X_m)$ --- the minimal degree of a rational function on $X_m$ --- is the following. The curve $X_m$ already admits a rational function $\pi_m: X_m \to \mP^1_k$ of degree $[S_\ell:S_m\times S_{\ell-m}]=\binom{\ell}{m}$. If $X_m$ were to admit a different rational function $\gamma: X_m \to \mP^1_k$ of degree $d$, then the map $(\pi_m, \gamma): X_m \to \mP^1_k \times \mP^1_k$ would realize $X_m$ as a curve of bidegree $(\binom{\ell}{m}, d)$ on the quadratic surface $\mP^1\times \mP^1$. Such a curve has genus at most $d \binom{\ell}{m}$. To calculate the genus $g_m$ of $X_m$ we simply apply the Riemann--Hurwitz formula to find that  
$g_m=1- \binom{\ell}{m}+(\ell-1)\binom{\ell}{m-1}$. Substituting into our bound $\gon (X_m)\geqslant g_m/\binom{\ell}{m}$ does not give a good bound on the gonality; for example, for $m=2$ the right hand side is less than $2$.

To fix this argument we consider the curves $Y_m\coloneq \tilde{X}/S_{\ell-m}$, which are degree $m!$ covers of $X_m$'s. The curves $Y_m$ form a tower $Y_m\to Y_{m-1}\to Y_{m-2}\to\dots\to Y_0=\mP^1$ corresponding to the nesting of groups $S_{\ell-m}\subset S_{\ell-m+1} \subset \dots \subset S_{\ell}$. Suppose $\gamma: Y_m \to \mP^1_k$ is a gonal map, that is, a rational function of degree $\gon(Y_m)$. Then using $\gamma$ we can map $Y_m$ into the algebraic surface $Y_{m-1}\times \mP^1$. A theorem of Castelnuovo--Severi can be used to estimate the genus of the curve on such a product surface; see \Cref{Castelnuovo}. Concretely, the Castelnuovo--Severi inequality shows that $\gon(Y_m) \geqslant \deg R/2(\ell-m+1)$, where $R$ is the ramification divisor of the covering $Y_m \to Y_{m-1}$; the quantity $\deg R$ appears in the Riemann--Hurwitz formula, $\deg R=\sum_{P\in Y_{m}}(e_P-1)$, where $e_P$ is the ramification index. A standard calculation using this formula gives $\deg R \geqslant (2\ell-2)\left((\ell-2)(\ell-3)\dots (\ell-m)\right)$, and so $\gon(Y_m)\geqslant ((\ell-2)\dots (\ell-m))(\ell-1)/(\ell-m+1)$.
Finally, since any degree $e$ rational function on $X_m$ gives, by composition, a degree $m!e$ rational function on $Y_m$, we conclude
\[\gon(X_m)\geqslant \gon(Y_m)/m! \geqslant \frac{(\ell-1)\binom{\ell-2}{m-1}}{m(\ell-m+1)}.\]
With this new bound, it is easy to see that $\min_{\ell/2\geqslant m\geqslant2}\gon(X_m)\geqslant (\ell-2)/2 $. Thus finiteness in Hilbert's irreducibility theorem holds for points of degree at most $(\ell-2)/4$ in this case.

Philosophically, the key to this approach is to study the gonality as a function on the diagram of curves $\tilde{X}/H$ for various subgroups $H$, rather than working directly with the curves $X_m$. Such a study is necessary for proving results like \Cref{thm:Galois-groups}, as points on $X_H$ parametrize fibers whose Galois group is contained in $H$. Note also that even in this nice special case the resulting bound $(\ell-2)/4$ is linear in the number of branch points $b=2\ell-2$ (compare with \Cref{thm:intro:S-HIT}). 

This core argument will be combined with the following additional methods. On the arithmetic side, we will replace the simple gonality inequality of \cite{abramovich1991abelian} and \cite{frey1994curves} with a more sophisticated recent version from \cite{kadets2025subspace}. We will also use the work of Vojta \cite{vojta1992generalization} and Song--Tucker \cite{song2001arithmetic} to sometimes avoid considering gonality altogether. On the group-theoretic side, the main challenge is to deal with the possibly complicated behavior of branch cycles. Here we will use the striking recent work of Burness--Guralnick \cite{BG} on indices of primitive permutation groups, which are not of product type. Combining their results with the Castelnuovo--Severi inequality to bound gonalities of resolvent curves leads to \Cref{prop:intro-monodromy}.
\subsection{Related work}
The original Hilbert--Siegel problem concerns  fibers of polynomial maps over points in $\mathbb Z$. An almost complete solution was announced recently \cite{BKN}. For  degree-$n$ maps $f:X\to\mP^1_\mQ$  with $g_X>0$ and general monodromy   $\Mon_\mQ(f)=S_n$, the fibers of $f$ are reducible only over finitely many  $a\in \Z$; see \cite{Mul4}, and \cite{muller2002finiteness} for similar results under other assumptions. 
For indecomposable $f\in \mQ[X]$ with $\deg(f)>20$,  there are only finitely many  rational $a\in\mQ\setminus f(\mQ)$ with reducible fibers under $f$,  by the combination of theorems of M\"uller \cite{Mul2} and Guralnick--Shareshian \cite{GS}; see \cite[Thm.\ 5.4]{konig2024reducible}. Similar results were shown for indecomposable rational functions $f\in\mQ(x)$ of sufficiently large degree in \cite{MN,PHDTali}, using \cite{neftin2024monodromy}, and \cite[\S 4]{bilu2026values}. 
Algebraic points of fixed degree $d$ in fibers of maps $f:X\to\mP^1_\mQ$ over rational points were also recently  considered  by Derickx--Rawson \cite{DR25}.

 In contrast, the analogues of Hilbert's irreducibility for higher-degree points have not been explored much. Tucker \cite{tucker2002irreducibility} has discovered that even a weak version of Hilbert's irreducibility cannot hold for all coverings:  for any $d \geqslant 4$ there is a covering of curves $\phi:X \to X_0$ such that for any finite extension $L/k$, there are at most finitely many points $P$ of degree $d$ on $(X_0)_L$ for which $\phi^{-1}(P)$ is irreducible. This uses the curves constructed by Debarre and Fahlaoui \cite{debarre1993abelian}; see Examples \ref{exam:weaker-vesionts}, \ref{exam:tucker}. Thus one can only hope to obtain results with restrictions on the types of coverings being considered. Bary-Soroker, Fehm, and Wiese have proved that the compositum of all degree $d$ extensions of $k$ is Hilbertian \cite{bary2016hilbertian}, though this is not directly connected to our results, as most elements in such a compositum are of arbitrarily high degree. Finally, the observation that the gonalities of curves which belong to a diagram should be studied simultaneously, as a function on the diagram, appeared in the work of Ellenberg, Hall, and Kowalski \cite{ellenberg2012expander}, who considered a very different setup: infinite towers of curves, with monodromy similar to that of modular curves. 

More broadly, problems which concern the arithmetic of higher-degree points on curves have a long history, and continue to be actively studied; see \cite{viray2024isolated} for a survey. 
\subsection{Structure of the paper}
In \Cref{sec:preliminaries} we introduce the geometric category of curves and their coverings, as well as discuss the main arithmetic invariant of curves --- the \emph{minimum density degree} --- and its connections to geometric features of curves, like gonality. In addition to classical tools like the Castelnuovo--Severi inequality, we also describe a few recent developments upon which we rely. 
In \Cref{Section:irreducibility-Sn} we study the irreducibility of fibers of coverings $\phi:X \to X_0$ with monodromy group $A_\ell$ or $S_\ell$, giving the proof of \Cref{thm:intro:S-HIT}. 
In \Cref{sec:Vojta} we consider the same problem as in \Cref{Section:irreducibility-Sn}, but in a different regime of parameters: we study coverings for which the number of branch points is higher than the degree, proving Theorems \ref{thm:large-branch} and \ref{intro:vojta}. 
In \Cref{sec:Galois-groups} we study the structure of Galois groups of fibers of an $A_\ell$ or $S_\ell$-covering, proving \Cref{thm:Galois-groups}. 
Finally, in \Cref{section:general-groups} we study indecomposable coverings with general monodromy group $G$, proving Theorems \ref{prop:intro-monodromy}, \ref{thm:product-main}. 

\subsection*{Acknowledgments} 
We thank Arno Fehm, Robert Guralnick, and Joachim K\"onig for providing helpful comments and references. 
The second author was funded by the Israel Science Foundation, grant no.~353/21. 
\section{Preliminaries}\label{sec:preliminaries} 
\subsection{Notation} 
Throughout we will work over a finitely generated field $k$ of characteristic zero. 
For an integer $m\geqslant1$, let $\binom{x}{m}$ denote the polynomial $x(x-1)\cdots (x-m+1)/m!$ and let $x^{\uline{m}}$ denote the polynomial $x(x-1)\cdots (x-m+1)$. 
\subsection{Curves}
 We will use the word \emph{curve} to denote a smooth integral projective variety $X/k$ of dimension $1$. Note that we do not require $X$ to be geometrically irreducible. All morphisms are considered over the field $k$. A {\it covering} is a finite morphism of curves. We will use certain notions (e.g., gonality) and results which are usually only introduced for geometrically irreducible $X$, and so in this section we briefly recall basic properties of this, more inclusive, class of curves. For a more comprehensive exposition of the basic theory of curves in this generality, see, for example, \cite[\href{https://stacks.math.columbia.edu/tag/0BRV}{Tag 0BRV}]{stacks-project}. 
 
 In the function field language, curves correspond to finitely generated field extensions $k(X)/k$ of transcendence degree $1$.
For any curve $X$, the field of constants  $L=k(X)\cap \overline{k}$ coincides with the field $H^0(X, \mathcal{O}_X)$. Then $L/k$ is a finite extension, and the structure morphism $X \to \Spec k$ factors as $X \to \Spec L \to \Spec k$. We write $X/L$ to refer to $X$ as a curve over $\Spec L$ in this fashion; this notation should not be confused with the base change $X_L$.  In particular, the base change $X_\mC$ of a curve to the complex numbers is a union of $[L:k]$ copies of $(X/L)_\mC$. The basic connections between $X$ and $X/L$ are summarized in the following proposition. Here the \emph{gonality} $\gon_k X$ of a curve $X/k$ is the minimal degree of a finite morphism to $\mP^1_k$. We use the notation $\chi(X)$ to denote the Euler characteristic $\chi(X)=\chi(X, \mathcal{O}_X)=\dim H^0(X, \mathcal{O}_X)-\dim H^1(X, \mathcal{O}_X)$; for a geometrically integral curve $\chi(X)=1-g$, where $g$ is the genus of the Riemann surface $X(\mathbb{C})$. 

\begin{prop}\label{prop:curve-trivialities}
Suppose $X/k$ is a curve and $L=H^0(X, \mathcal{O}_X)$. Then the following hold:
\begin{enumerate}
    \item\label{item:prop:eulerch} $\chi(X)=[L:k] \chi(X/L)$;
    \item\label{item:prop:pointdeg} Let $x$ be a closed point of $X$ and $\tilde{x}$ be the corresponding point on $X/L$. Then $\deg x = [L:k]\deg \tilde{x}$;
    \item\label{item:prop:curvecat} There is an equivalence of categories between curves (and finite morphisms, over $k$) and finitely generated field extensions $K/k$ of transcendence degree $1$;
    \item\label{item:prop:RH} If $\phi\colon X \to Y$ is a finite covering, then the Riemann--Hurwitz formula holds
\[2\chi(X)=2\deg(\phi) \chi(Y)-\deg R_\phi,\]
where $R_\phi$ is the ramification divisor: $\mathrm{supp} R_\phi=\mathrm{supp}\ \omega_X/\phi^{*}\omega_Y$\footnote{here $\omega_X$ is the sheaf of differential one-forms, $\omega_X=\Omega^1_X$.}, and the multiplicity at a point $P \in \mathrm{supp}\ R_\phi$ is $\mathrm{length}_P(\omega_X/\phi^*\omega_Y)$;
    \item\label{item:prop:gonality} The gonalities of $X$ and $X/L$ satisfy
    \[\gon_k (X) = [L:K]\gon_L(X/L).\]
\end{enumerate}
\end{prop}
\begin{proof}
Claim \eqref{item:prop:eulerch} (resp.\ \eqref{item:prop:pointdeg}) comes from the equality $\dim_k V = [L:k]\dim_L V$ for the $L$-vector spaces $V=H^i(X,\mathcal O_X)$, $i=0,1$ (respectively for the residue fields of $x$ over $k$ and $L$). Statements \eqref{item:prop:curvecat} and \eqref{item:prop:RH} are part of the general theory of curves (over perfect fields); see  \cite[\href{https://stacks.math.columbia.edu/tag/0BRV}{Tag 0BRV}]{stacks-project}, and more specifically \cite[\href{https://stacks.math.columbia.edu/tag/0C1B}{Tag 0C1B}]{stacks-project} for part (4). Claim \eqref{item:prop:gonality} follows from the universal property of the fiber product $\mP^1_L=\mP^1_k \times_k\mathrm{Spec}\ L$: since $X$ has a structure morphism $X\to\Spec L$, the morphisms $X\to \mP^1_k$ over $k$ are in bijection with the morphisms $X \to \mP^1_L$ over $L$.
\end{proof}
\begin{rem}
    Part of the expression for the Riemann--Hurwitz formula is a formula for $\deg R_\phi$ as a sum of local terms; we will come back to this form of the Riemann--Hurwitz formula in Section \ref{Section:permutations-and-RH}.
\end{rem}

\subsection{Gonality and resolvent curves}\label{sec:resolvent}
We begin by describing some well-known generalities on interpreting the finiteness in the Hilbert irreducibility theorem as a question on arithmetic of the ``resolvent'' curves. Suppose $\phi: X \to Y$ is a covering of curves of degree $\ell$ with (permutation) Galois group $G \subset S_\ell$, and let $\tilde{X} \to Y$ be its Galois closure. In the language of \'etale fundamental groups, $G$ is the image of the morphism $\pi_1^{\et}(Y\setminus B_\phi, \bar{y})\to S_\ell$ which corresponds to the covering $\phi$, where $B_\phi$ is the branch locus; this image is also known as the \emph{monodromy group} of $\phi$. For every subgroup $H \subset G$, let $X_H\coloneqq \tilde{X}/H$ denote the $H$-resolvent curve; it is a degree $[G:H]$ covering of $Y$. The isomorphism class of $\phi_H: X_H \to Y$ depends only on the conjugacy class of $H$. If $P$ is a closed point of $Y$ which is not in the branch locus, then the fiber $\phi^{-1}(P)$ is \'{e}tale over $P=\Spec k(P)$, and so defines a permutation action $\Gal_{k(P)} \to S_\ell$; this action factors through $\varphi_P: \Gal_{k(P)} \to G$, so that $\Gal(k(\phi^{-1}(P))/k(P))=\Im \varphi_P$. We shall repeatedly use the identification between the orbits of $\Gal_{k(P)}$ on $G/H$, acting via $\varphi_P$, and the irreducible components of the fiber $\phi_{H}^{-1}(P)$. In particular, the fiber $\phi^{-1}(P)$ is irreducible if and only if the action $\varphi_P$ is transitive, and it has full Galois group if and only if $\varphi_P$ is surjective. So the irreducibility of the fiber $\phi^{-1}(P)$ over $k$, which is equivalent to the transitivity of the action $\varphi_P$, is in turn equivalent to the action of $\Gal_{k(P)}$ on $m$-sets having no fixed points for all $m$. For any permutation group $G \subset S_\ell$ we will make no distinction between the action of $G$ on $m$-sets and the action on $G/(G\cap (S_m\times S_{\ell-m}))$.

The key property of the resolvent curves $X_H$ is that $\mathrm{im}(\varphi_P)$ is contained in a conjugate of $H$ if and only if the fiber $\phi_H^{-1}(P)$ contains a $k(P)$-rational point. We will show that, under some conditions on the covering, the curves $X_H$ have only finitely many points of small degree. First we summarize some known results on the finiteness of the set $|X|_d$ of points of degree $d$ (over $k$)  on a curve $X$.

Suppose $k$ is a finitely generated field of characteristic zero, and $X/k$ is a curve. Following \cite{kadets2025subspace} and \cite{viray2024isolated} we define an integer called the \emph{minimum density degree} of $X$ by: 
\[\min \delta(X)\coloneqq \min\{d\colon \#|X|_d=\infty\}.\]

Intuitively, one should think of $\min \delta(X)$ as analogous and closely related to the \emph{gonality} $\gon_k X$ of the curve $X$. The main results of this paper can be understood as asserting lower bounds for $\min \delta(X_H)$, where $X_H$ are resolvent curves of a rational function $X \to \mP^1_k$ (or a more general covering of curves). The resolvent curves are naturally equipped with a rational function of degree $[G:H]$, but usually have no ``natural'' sources of lower-degree rational functions, suggesting that the gonality and $\min\delta(X_H)$ should grow quickly in the diagram of resolvent curves. The following simple lemma will be often used to trade studying $\min\delta(X_H)$ for $\min\delta(X_{H'})$, where $H, H'$ are close in the subgroup lattice.

\begin{lem}\label{lem:gon-and-airr}
Suppose $\phi: X \to Y$ is a covering of curves. Then 
\begin{align*}
    \min\delta(Y)\leqslant   \min\delta(X) \leqslant &  \deg \phi\cdot \min\delta(Y), \text{ and }\\
    \gon_k Y \leqslant  \gon_k X \leqslant  & \deg \phi \cdot \gon_k Y.
\end{align*}
\end{lem}
\begin{proof}
The first line of inequalities follows by observing that a pushforward of a degree $d$ point has degree at most $d$, and a pullback has degree at most $d \cdot\deg \phi$.

In the second line, the inequality $\gon_k X \leqslant \deg \phi \cdot \gon_k Y$ follows by considering the composition of $\phi$ with the gonal map of $Y$. What is left to show is the well-known, but not obvious inequality 
\[\gon_k X \geqslant \gon_k Y.\]

We first observe that by enlarging $k$ to be $H^0(Y,\mathcal{O}_Y)$, we can, without loss of generality, assume that $Y$ is geometrically integral. Let $L$ denote the field $H^0(X, \mathcal{O}_X)$. Note that $\gon_L Y_L \geqslant [L:k]^{-1}\gon_k Y$: if $f \in L(Y)^{\times}$ is a gonal map, then $\mathrm{Nm}_{L(Y)/k(Y)}(f) \in k(Y)^\times$ is a rational function on $Y/k$ of degree at most $[L:k]\deg f$. Therefore, by taking a base change to $L$, and using part \eqref{item:prop:gonality} of Proposition \ref{prop:curve-trivialities}, we  reduce to the case of both $X$ and $Y$ being geometrically integral over $k$. In this case the conclusion is standard. We give a sketch of a geometric proof below, and refer to \cite[Proposition A.1.vii]{poonen2007gonality} for a function-field treatment.

Since $X$ has gonality $\gamma\coloneqq \gon_k X$, there is a family of degree $\gamma$ effective divisors on $X$ parameterized by the projective line; in other words there exists a nonconstant morphism $\mP^1_k \to \mathrm{Sym}^{\gamma} X$. The map $\phi$ induces the quasifinite ``pushforward of divisors'' map $\mathrm{Sym}^{\gamma} X \to \mathrm{Sym}^{\gamma} Y$. Therefore $\mathrm{Sym}^{\gamma} Y$ contains a rational curve $S$.  The Abel--Jacobi map $\mathrm{Sym}^{\gamma} Y\to \mathrm{Pic}^{\gamma}_Y$ is necessarily trivial when restricted to $S$, since it is a map from a rational curve to  an abelian variety. If $s_1, s_2 \in S(k)$ are two distinct points, then the degree $\gamma$ effective divisors they represent are linearly equivalent, implying that there exists a map $Y \to \mP^1_{k}$ of degree at most $\gamma$.
\end{proof}

We need simple geometric conditions on curves $X$ that enforce finiteness of the sets of degree $d$ points $|X|_d$ on $X$ (over $k$), similar to the condition $g\geqslant 2$, which by Faltings' theorem guarantees that $|X|_1=X(k)$ is finite. Unfortunately, no complete analogue of Faltings' theorem for higher-degree points is currently known. However, we will get by with the following proposition.
\begin{prop}\label{prop:arithm-gonality}
    Suppose $X$ is a curve over $k$. Then the following hold:
    \begin{enumerate}
        \item The \emph{gonality} $\gon_k X$ of $X$ is at most $2\min \delta (X)$;
        \item If $d=\min \delta(X)$, then for some integer $e|d$, the curve $X$ admits a degree $e$ covering $\phi:X \to Y$ to a curve $Y$ satisfying $\min \delta(Y)=d/e$ and
        \[-\chi(Y) \leqslant \frac{3}{4} \left(d/e\right)^2 + d/e-1\]
    \end{enumerate}
\end{prop}
\begin{proof}

Let $L=H^0(X, \mathcal{O}_X)$. By parts \ref{item:prop:pointdeg} and \ref{item:prop:gonality} of Proposition \ref{prop:curve-trivialities}, viewing $X$ as a curve over $L$, we see that $\gon_L X=[L:k]^{-1} \gon_k X$ and $\min \delta(X/L)=[L:k]^{-1}\min \delta(X)$. Therefore, to prove the first claim, it is enough to consider the case of geometrically irreducible $X$. In this case the first claim is implied (indirectly) by the results of Abramovich and Harris \cite{abramovich1991abelian}, and is explicitly stated by Frey in \cite[Proposition 2]{frey1994curves}.

Similarly, the second claim reduces to the case of geometrically integral $X$. It then follows from \cite[Theorem 1.3]{kadets2025subspace} by repeatedly applying the theorem, as we now explain. The theorem states that for any geometrically integral curve $X_1$ with $\min \delta(X_1)=d_1$ either the genus of $X_1$ is bounded, $g(X_1)\leqslant f(d_1)$ (for an explicit function $f$ which we give below), or there exists a covering $\phi_{12}\colon X_1 \to X_2$ such that $\min \delta(X_2)=\deg \phi_{12} \min\delta(X_1)$. If the latter holds, we can apply the theorem to the curve $X_2$, and continue in the similar manner, obtaining a tower of covers $X_1 \to X_2 \to \cdots X_n$ such that, if $\phi_{1,n}:X_1 \to X_n$ is the composite map, then $\min \delta(X_1)=\deg \phi_{1,n}\min \delta(X_n).$ Such a tower has to eventually terminate, and so we can assume that the genus of $X_n$ is bounded $g(X_n)\leqslant f(d_1/\deg \phi_{1,n})$. In our case we take $X=X_1$, $Y=X_n$, $\phi=\phi_{1,n}$. Then all that is left to check is that for any positive $u$, $f(u)\leqslant \frac{3}{4} u^2 + u$ (recall that $-\chi(Y)=g(Y)-1$.) This is readily seen from the explicit formula for $f$, namely 
\[f(u)=\max \left(\frac{u(u-1)}{2}+1, 3m(m-1)+m\epsilon \right),\]
where $m\coloneqq \lceil u/2\rceil-1$ and $\varepsilon \coloneqq 3u-1-6m<6.$

We note here that the results of \cite{kadets2025subspace} are only stated over number fields. However the proofs only use the Mordell--Weil theorem,  Hilbert's irreducibility theorem, and  the characteristic zero assumption. It therefore  carries over without change to a finitely generated field of characteristic zero using  the Lang--N\'eron theorem \cite{lang1959rational}; see also \cite{conrad2006chow} for a modern exposition. These results have also been recently extended to other function fields, with some caveats concerning constant curves, in \cite{van2026points}.
\end{proof}

To estimate gonality we will use the following corollary of the Castelnuovo--Severi inequality. Recall that a covering of curves $\phi:X \to Y$ is called \emph{indecomposable} if any factorization $X \to Z \to Y$ of $\phi$ through another curve $Z$ satisfies $X=Z$ or $Z=Y$.

\begin{prop}\label{Castelnuovo}
    Suppose $\phi: X \to Y$ is an indecomposable covering of curves. Then 
    \[\gon(X) \geqslant \min\left\{\deg \phi \cdot \gon (Y), \frac{\deg \phi\  \chi(Y)-\chi(X)}{\deg \phi - 1}\right\}.\]
\end{prop}
\begin{proof}
    First note that by \Cref{prop:curve-trivialities}, without loss of generality, we can assume that $Y$ is geometrically irreducible.
    Suppose $\gamma: X \to \mP^1$ is the gonal map, and consider the morphism $\psi=(\phi,\gamma): X \to Y \times \mP^1$. The image of $\psi$ is birational to a curve $Z$ such that $\phi$ factors as $X \to Z \to Y$. Since $\phi$ is indecomposable, either $Z=Y$ or $Z=X$. If $Z=Y$, then $\gamma$ factors through $\phi$, giving $\gon (X) \geqslant \deg \phi\cdot \gon (Y)$. If $Z=X$, then $\psi$ is birational onto its image. In this case we can appeal to the Castelnuovo--Severi inequality \cite{kani1984castelnuovo}. The inequality states that for any one-dimensional geometrically irreducible variety $\tilde{X}$ on the product surface $\tilde{Y}\times \mP^1$, such that the degrees of the projections of $\tilde{X}$ on $\tilde{Y}$ and $\mP^1$ are $e$ and $f$ respectively, the geometric genus of $\tilde{X}$ satisfies\footnote{This is usually written in terms of the genera of $\tilde{X}, \tilde{Y}$ and $g(\mP^1)=0$ as $g(\tilde{X})\leqslant eg(\tilde{Y})+fg(\mP^1)+(e-1)(f-1)$.} 
  \[-\chi(\tilde{X})\leqslant e(-\chi(\tilde{Y}))+ f(e-1).\]
Now let $L=H^0(X, \mathcal{O}_X)$ and consider the base change $X_\mC$ which is a union of $[L:K]$ irreducible curves,  each isomorphic to $(X/L)_\mC$. Applying the Castelnuovo--Severi to the base-changed curves $\tilde{X}=(X/L)_\mC$ and $\tilde{Y}=Y_\mC$, gives

\[-\chi((X/L)_\mC)\leqslant \frac{\deg \phi}{[L:k]}(-\chi(Y)) + \frac{\gon (X)}{[L:k]} (\deg \phi/[L:k] - 1).\]

Since $-\chi((X/L)_\mC)=-\chi (X/L)=-[L:k]^{-1}\chi(X/k)$ by  \Cref{prop:curve-trivialities}, we have

\[\gon (X) \geqslant \frac{\deg \phi\cdot \chi(Y)-\chi(X)}{(\deg \phi/[L:k]) - 1}\geqslant \frac{\deg \phi\cdot \chi(Y)-\chi(X)}{\deg \phi - 1},\]
as claimed.    
\end{proof}
\begin{rem}
    Note that the second term in the minimum of Proposition \ref{Castelnuovo} is exactly $\deg R_\phi/(2\deg \phi - 2)$; this shows that gonality has to grow in sufficiently ramified covers.
\end{rem}

The same argument can be applied to give the following bound for maps which are not required to be indecomposable.
\begin{lem}\label{rem:Cast-generalized}
Let $\phi:X\to Y$ be a covering of curves and $e=\min\deg \phi_2$ be the minimal degree of a right composition factor $\phi_2$ of $\phi=\phi_1\circ\phi_2$. Then $$\gon(X)\geqslant \min\left\{e\cdot\gon(Y),\frac{\deg R_\phi}{2\deg \phi-2}\right\}.$$
\end{lem}
\begin{proof}
    Suppose $\gamma:X \to \mP^1$ is the gonal map, and consider the map $\psi=(\phi, \gamma)\colon X \to Y \times \mP^1$. If $\psi$ is not birational onto its image, let $Z$ be the normalization of the image of $\psi$. Since $Z$ is a covering of $Y$, we have $\gon(Z) \geqslant \gon(Y)$ by \Cref{lem:gon-and-airr}. Therefore $\deg \gamma \geqslant  \deg \psi\ \gon(Z) \geqslant e\cdot \gon(Y)$, since both $\phi$ and $\gamma$ factor through $\psi.$ If $\psi$ is birational onto its image, then we can apply the Castelnuovo--Severi inequality to $\psi(X)$. This gives, by the same algebra as in \Cref{Castelnuovo}, the inequality 
    \[\gon(X)\geqslant  \frac{\deg \phi\ \chi(Y)-\chi(X)}{\deg \phi - 1}=\frac{\deg R_\phi}{2\deg \phi -2}. \qedhere\]
\end{proof}

\subsection{Coverings, permutation groups, and the Riemann--Hurwitz formula}\label{Section:permutations-and-RH}
Our main goal will be to estimate the genus, gonality, and minimum density degree of various curves $X$ covering a geometrically integral curve $X_0$. The tools we have introduced in the previous section can be interpreted from a group-theoretic point of view in the language of permutation groups and branch cycles. We will use the following common notation for combinatorics of permutations.

\begin{defn}
    Suppose $x \in S_\ell$ is a permutation. Let $f(x)$, $\orb(x)$, $\ind(x)$ denote the number of fixed points, number of orbits, and the \emph{index} $\ind(x)=\ell-\orb(x)$, respectively. 
        For $G\leq S_\ell$, the \emph{minimal index} $\ind(G)$  is the minimum of $\ind(x)$ over $x\in G\setminus \{1\}$.
\end{defn}

We will use known group-theoretic inequalities on the indices, specifically for permutations which arise as a branch cycle of a covering of curves --- a notion that we now explain.
Suppose $\phi:X\to X_0$ is a covering of curves of degree $\ell$ with monodromy group $G \subset S_\ell$, and suppose $X_0$ is geometrically irreducible. For any closed point $P$ of $X_0$, the fraction field of the completed local ring at $P$ gives a map $\mathrm{Spec}\ k(P) (\!(t)\!) \to X_0$; taking the fiber under $\phi$ gives a morphism $\hat{\Z}=\Gal_{\overline{k(P)}(\!(t)\!)} \to G$, defined up to conjugation. A generator of the image of this morphism is called a \emph{branch cycle} $\gamma_P$ of $P$. A point $P$ is called a \emph{branch point} if $\gamma_P \neq id$.  Geometrically, the branch cycle describes the topological action of a loop around $P$ on the fiber of $\phi$ at a point near $P$. If $X$ is not geometrically irreducible, the base change $X_\mC$ is a union of $[L:k]$ isomorphic components, and the monodromy action on each of the components is the same. This allows us to state the Riemann--Hurwitz formula, from the theory of Riemann surfaces, in the more general case as follows: 
\begin{equation}\label{equ:RH}2(\deg \phi\ \chi(X_0)-\chi(X)) = \deg R_\phi = \sum_{P\in |X_0|} \deg P\ \ind(\gamma_P).\end{equation}
Note that if $\gamma_P$ acts with orbits of sizes $e_1, \dots, e_s$, then $$\ind \gamma_P=n-s=\sum_ie_i - s =\sum_i(e_i-1);$$ the latter expression is the more familiar Riemann--Hurwitz formula in terms of the ramification indices. 
The expression on the right hand side of \eqref{equ:RH}, equal to $\deg R_\phi$, is called the \emph{Riemann--Hurwitz contribution} (as it measures the contribution of ramification to the genus growth from $X_0$ to $X$).
It is more convenient for us to sum over points $P\in X_0(\mC)$ rather than over closed points; similarly when we refer to the ``number of branch points'' this should be interpreted as a count of branch points in $X_0(\mC)$, rather than the number of closed points in the branch locus. 
\begin{rem}\label{rem:identify}
Given a transitive action of $G$ on a set $\Delta$ with stabilizer $H$,
and a covering of curves $\phi$ of monodromy $\Mon(\phi)=G$ and stabilizer $H$, we shall often use the (noncanonical) identification between the action of $\Mon(\phi)$ on a fiber of $\phi$ with its action on $\Delta\cong G/H$. In particular, we shall  often  deduce  the ramification indices over a branch point $P$  from  the lengths of orbits of $\gamma_P$ on $\Delta$, and vice versa.
\end{rem}

The following is a combination of Proposition \ref{Castelnuovo} and the Riemann--Hurwitz formula. 
\begin{prop}\label{prop:min-ind}
    Let $\phi:X\to X_0$ be an indecomposable covering of degree $n$, monodromy group $G\subset S_n$, with $b$ branch points, and minimal index $\ind(G)$. Then \[\gon(X)\geqslant\min\left\{n\cdot\gon(X_0),\frac{b\cdot \ind(G)}{2(n-1)}\right\}.\] 
\end{prop}
\begin{proof}
    Let $\deg R_\phi$ denote the total Riemann--Hurwitz contribution for $\phi$. 
    Since each branch point contributes at least $\ind(G)$ to $\deg R_\phi$, we have $\deg R_\phi\geqslant b\cdot\ind(G)$. 
    By Proposition~\ref{prop:arithm-gonality}, 
    $$\gon (X)\geqslant \min\left\{n\cdot \gon(X_0),\frac{\deg R_\phi}{2(n-1)}\right\}\geqslant \min\left\{n\cdot \gon(X_0),\frac{b\cdot \ind(G)}{2(n-1)}\right\}. \qedhere$$
\end{proof}
   
Indecomposable coverings admit primitive monodromy groups $G$. Recall that a group $G$ acting on a set $\Delta$ is {\it primitive} if it does not act on any proper nontrivial partition of $\Delta$.

Given a primitive group $G$ acting on a set $\Delta$ and a transitive group $H\leqslant S_t$, let $G \wr H=G^t\rtimes H$ denote their  wreath product, where the semidirect action permutes the $t$ copies of $G$. 
    The product-type action of $S_\Delta \wr S_t$ is its standard degree $\ell^t$ action on the set $\Delta^t$. 
    Burness--Guralnick \cite{BG} provide good bounds on indices of elements for primitive groups with the exception of groups  $A_\ell^t\leqslant G\leqslant S_\ell\wr S_t$, $\ell\geqslant 5$, $t\geqslant 1$ of product-type action with respect to the $S_\ell$-action on $m$-sets, that is,  on the set $\Delta$ of cardinality-$m$ subsets of $\{1,\ldots,\ell\}$ for some $1\leqslant m\leqslant \ell/2$. 
For $t>1$, such groups $G$ are called  non-basic in the terminology of the Aschbacher--O'Nan--Scott theorem \cite[\S4.3]{Cameron}, and  non-elemental in that of \cite{VdW}. 
\begin{rem}\label{rem:indec} 
The primitivity of $G$ implies that its image in $S_t$ is transitive \cite[Lemma 2.7A]{dixon1996permutation}. Conversely for $t\geqslant 2$ and $\ell\geqslant 3$,  \cite[Lemma 2.11]{NZ2}\footnote{The assumption $\ell\geqslant 5$ is made in \cite[Lemma 2.11]{NZ2} but is  not used in the proof of this direction.} implies: If $G\leqslant S_\ell\wr S_t$ contains $A_\ell^t$ and the image of $G$ in $S_t$ is transitive, then $G$ is primitive. 
\end{rem}

Applying the index bounds of  \cite{BG} to Proposition \ref{prop:min-ind}, we get the following consequence. 
As it uses \cite[Thm.\ 7]{BG}, it relies on the classification of finite simple groups. 
\begin{cor}\label{cor:BG}
    Let $\phi:X\to X_0$ be an indecomposable cover of degree $n$,  monodromy group $G$, and $b$ branch points. Then either
    $\gon(X)\geqslant \min\{n\cdot\gon(X_0),3b/28\}$ and $\min\delta(X)\geqslant \min\{\gon(X_0)\cdot n/2,3b/56\}$, or 
    $A_\ell^t\leqslant G\leqslant S_\ell\wr S_t$, $\ell\geqslant 5$, $t\geqslant 1$ is of product type with respect to the $S_\ell$-action  on $m$-sets for   $1\leqslant m\leqslant \ell/2$. 

    If, moreover, $G=A_\ell$ or $S_\ell$ in an action that is different from its action on $m$-sets, then $\gon(X)\geqslant \min\{n\cdot \gon(X_0),b/8\}$ and $\min\delta(X)\geqslant \min\{\gon(X_0)\cdot n/2, b/16\}$. 
\end{cor}
\begin{proof}
Assume the action of $G$ is not one of the above product-type actions. 
Then  $\ind(G)\geqslant 3n/14$, and even $\ind(G)\geqslant n/4$ if $G=A_\ell$ or $S_\ell$, in an action different from that on $m$-sets, 
by \cite[Theorem 7]{BG}.  
  Proposition \ref{prop:min-ind} then gives:
  \begin{align*}
      \gon(X)&\geqslant \min\left\{n\cdot \gon(X_0), \frac{3}{28}b\right\},\text{   and even }\\
     \gon(X)& \geqslant \min\left\{n\cdot \gon(X_0), \frac{b}{8}\right\} \text{ for }  G \in \{A_\ell, S_\ell\}.
  \end{align*}

Furthermore, as $\min\delta(X)\geqslant \gon(X)/2$
by Proposition \ref{prop:arithm-gonality}.(1), we have: 
  \begin{align*}
      \min \delta(X)&\geqslant \min\left\{\frac{n}{2}\cdot \gon(X_0), \frac{3}{56}b\right\}, \text{ and even }\\
     \min \delta(X)& \geqslant \min\left\{\frac{n}{2}\cdot \gon(X_0), \frac{b}{16}\right\} \text{ for }  G \in \{A_\ell, S_\ell\}.\qedhere
  \end{align*}

\end{proof}
\Cref{cor:BG} will be used to reduce the analysis of Hilbert irreducibility property for arbitrary indecomposable covers $\phi:X\to X_0$, to an analysis of covering with monodromy groups of special type. 
\subsection{Examples}\label{sec:examples}

In this section we collect examples of situations in which various versions of Hilbert's irreducibility fail.

\begin{exam}[Covers with few branch points]\label{exam:few-branch-points}
    Consider the polynomial $\phi(x)=x^a(x-1)^b\in\mQ[x]$ for coprime $a,b\geqslant 1$ as a map $\phi:X_1\to X_0$ of degree $\ell=a+b\geqslant 5$, where $X_i= \mP^1_\mQ$, $i=0,1$. This map has three branch points $0, a/(a+b), \infty$ which is less than the minimal number required in Theorem \ref{prop:intro} and in the more general Corollary \ref{cor:hilb-for-poly} below. We claim that the fibers are reducible over infinitely many degree $d$ points $P\notin \phi(|X_1|_d)$ for all $   d\geqslant 1$. 

Let $\tilde X$ denote the Galois closure of $\phi$. Then a direct calculation shows that $\Mon_\mC(\phi)=\Mon_\mQ(\phi)=S_\ell$, and 
the quotient $X_2\coloneqq \tilde X/(S_{\ell-2}\times S_2)$ by a two-point stabilizer is of genus $0$; see\ \cite[\S4]{neftin2024monodromy}. Moreover, $X_2(\mQ)\neq\emptyset$ (the unordered pair $\{0,1\}\in X_2(\mC)$ corresponds to a rational preimage of $0\in \mP^1(\mQ)$), and hence $X_2\cong \mP^1_\mQ$. Let $\phi_2:X_2 \to X_0$ denote the natural degree ${\ell\choose 2}$ map $\tilde X/(S_{\ell-2}\times S_2)\to \tilde X/S_\ell$. 

To see the assertion consider the curve $Y_2=\tilde X/S_{\ell-2}$. It is well-known that\footnote{To see this, one may construct a Hilbert subset of $\Sym^dX_2(\mQ)$ satisfying the desired property as the intersection $U\cap V\cap (\cap_{i=1}^{d-1}W_i)$ of the following Hilbert subsets:  $U$ is the Hilbert subset of points with irreducible fiber under the map $\Sym^d\tilde X\to \Sym^d X_2$ induced by $\tilde X\to X_2$; the open subset $V$ is the preimage under $\Sym^dX_2\to\Sym^dX_0$ of the complement $x_i\neq x_j$, $1\leqslant i<j\leqslant d$ of the fat diagonal in $\Sym^d X_0$; and $W_i$ is the complement of the image of $(\Sym^{i}X_2\times\Sym^{d-i}X_2)(\mQ)\to \Sym^dX_2(\mQ)$ for $i=2,\ldots,d-1$.  } 
there are infinitely many points $Q\in |X_2|_d$ for which the fiber of the map $\tilde X\to X_2$ above $Q$ is irreducible and such that $\phi_2(Q)$ is still of degree $d$. 
For such $Q$, we get that $P=\phi_2(Q)\in \phi_2(|X_2|_d)$ is unramified and the fiber above $P$ has Galois group $S_{\ell-2}\times S_2\leqslant \Mon(\phi)$. The fiber $\phi^{-1}(P)$ then corresponds to the two orbits of $S_{\ell-2}\times S_2$ on $S_\ell/S_{\ell-1}\cong \{1,\ldots,\ell\}$. Moreover, since the orbits are of length $2,\ell-2$, these two preimages $P_1,P_2$ are of residue degrees $2,\ell-2$ and hence $P\notin \phi(|X_1|_d)$. 

Similar examples can be obtained from each of the examples in \cite[Table 4.1]{neftin2024monodromy} and for groups of product type $S_\ell\wr S_2$ from \cite[Theorem 1.2]{NZ2}. 
\end{exam}

\begin{exam}[Weaker versions of Hilbert's irreducibility]\label{exam:weaker-vesionts} 
   Suppose $\phi\colon X\to \mP^1_k$ is a covering. Then it is easy to prove that there exists infinitely many irreducible fibers of $\phi$ above degree $d$ points for any $d$. Indeed, $\phi$ induces a covering $\phi^{(d)}\colon\mathrm{Sym}^d X \to \mathrm{Sym}^d \mP^1=\mP^d$ of symmetric powers. Outside of a thin subset, rational points in the target $\mP^d$ correspond to degree $d$ points on $\mP^1_k$. This allows us to conclude by the usual Hilbert's irreducibility theorem.

    On the other hand, for a covering of curves $\phi:X \to X_0$ in general such a result does not always hold. The basic example here is a suitable isogeny between a pair of elliptic curves. For a concrete example, let $E, E'$ be the two elliptic curves in the isogeny class \cite[\href{https://www.lmfdb.org/EllipticCurve/Q/121/b/}{121.b}]{lmfdb}. These curves are connected by an $11$-isogeny $\phi$, and a direct calculation shows that $\phi: \Z=E(\mQ)\to E'(\mQ)=\Z$ is an isomorphism. Therefore, every fiber of $\phi$ is reducible. This type of examples is the reason behind the introduction of the weak Hilbert property by Corvaja and Zannier \cite{corvaja2017hilbert}; see also \cite{corvaja2022distribution} and the recent survey \cite{fehm2025hilbert}.
\end{exam}
\begin{exam}[Tucker's construction]\label{exam:tucker}
    The previous example of isogenies of elliptic curves is not stable under extensions of the base field. Tucker \cite[Theorem 3.3]{tucker2002irreducibility} gave a remarkable example, which shows that, even potentially, Hilbert's irreducibility for degree $d$ points cannot hold without additional assumptions. In other words, there exists a covering of curves $X\to X_0$ such that for every finite extension $L/k$ there are at most finitely many degree $d$ points on $(X_0)_L$ with irreducible fibers. This uses the curves constructed by Debarre and Fahlaoui in \cite{debarre1993abelian}.
\end{exam}

\begin{exam}[$A_\ell$-fibers of $S_\ell$-coverings]\label{ex:Al-Sl}
    \Cref{thm:Galois-groups} does not guarantee that an $S_\ell$-covering will have many $S_\ell$-fibers; rather, it only shows that many fibers must contain $A_\ell$. The reason for this comes from the following type of examples. Consider a degree $\ell$ covering of $\mP^1$ which has $n$ branch points whose branch cycle is a $3$-cycle, and two ramified fibers each containing a simple double point. Using Riemann's existence theorem, it is easy to make such a covering with monodromy $S_\ell$, simply by specifying a suitable surjection $\pi_1(\mP^1_\mathbb{C}\setminus\{b_1, \dots, b_{n+2}\}) \to S_\ell$ and descending to a finitely generated field $k$. If $\tilde{X}\to \mP^1$ is the corresponding $S_\ell$-Galois covering, then $\tilde{X}/A_\ell\to \mP^1_k$ is a degree $2$ covering with only two branch points. Thus $\tilde{X}/A_\ell$ is of genus zero, which means that there are infinitely many points in $\mP^1_k$ of degree $d$, for which the Galois group of the fiber is $A_\ell$.  
\end{exam} 

\section{Irreducibility and gonality in the action on sets}\label{Section:irreducibility-Sn}

In this section we analyze the arithmetic of fibers of a covering $\phi: X \to X_0$ of degree $\ell$ and monodromy $G$ equal to $A_\ell$ or $S_\ell$. Our main result, \Cref{thm:S-HIT}, shows that if $\phi$ has sufficiently many branch points, then for all but finitely many degree $d$ points $P \in |X_0|$ the fiber $\phi^{-1}(P)$ is irreducible. The results of this section apply, in particular, in the numerical range which corresponds to high-degree rational functions on a fixed curve $X$.

The irreducibility of the fibers is established by controlling $\min \delta(X_m)$ for the curves $X_m$ corresponding to the subgroup $G \cap (S_m \times S_{\ell-m})$ via the Galois correspondence (as in \Cref{sec:resolvent}). In terms of permutation actions, the covering $X_m$ corresponds to the action of $G$ on  $m$-sets of points in the fiber of $\phi$, as in \Cref{rem:identify}. Similarly, the subgroup $G \cap S_{\ell-m}$ produces a curve $Y_m$ that corresponds to the action of $G$ on ordered $m$-tuples. There is a natural covering $Y_m \to X_m$ of monodromy group $A_m$ or $S_m$. 

To summarize, in this section we will assume the following.

\begin{setup}\label{main-setup}
Suppose $\phi: X \to X_0$ is a degree $\ell$ covering of curves with monodromy group $G=A_\ell$ or $S_\ell$. Let $X_m$ denote the curve corresponding to the subgroup $G\cap (S_m \times S_{\ell-m})$ and $Y_m$ denote the curve corresponding to the subgroup $G \cap S_{\ell-m}$.
\end{setup}

We will need to compute the genus of $Y_m$, and for this purpose it is easier to work with branch points of $\phi$ with particularly simple ramification patterns. 

\begin{defn}
    Suppose $\phi: X \to X_0$ is a degree $\ell$ covering of curves, and $p \in X_0$ is a point. Given $\epsilon>0$ we say that $p$ is \emph{$\epsilon$-ramified} if it is a branch point and at least $(1-\epsilon)\ell$ points in $\phi^{-1}(p)$ are unramified.
\end{defn}
\begin{rem}
    Note that if $\epsilon$ is less than $2/\ell$, then $\epsilon$-ramified points are in fact unramified. For this reason the condition $\epsilon\geqslant 2/\ell$ appears in many of the theorems. Also, the definition implies that a fiber above an $\epsilon$-ramified point contains at least $\lceil (1-\epsilon)\ell \rceil$ unramified points.
\end{rem}

The following lemma shows that if a covering has many branch points, then it also has many $\epsilon$-ramified points.
\begin{lem}\label{lem:mystery}
    Suppose $\phi:X \to X_0$ is a degree $\ell$ covering of curves. Suppose that the number $b$ of branch points of $\phi$ satisfies 
    \[b\geqslant \frac{4}{\epsilon} \left(\frac{-\chi(X)}{\ell}+\chi(X_0)\right) + N\text{ for some }N\in\mathbb N.\] 
    Then the number of $\epsilon$-ramified branch points $p \in X_0$ is at least $N$.
\end{lem}
\begin{proof}
    This is a direct corollary of the Riemann--Hurwitz formula, as we now show. If $R_\phi$ is the ramification divisor on $X$, then each fiber with at least $\epsilon \ell$ ramified points (counting multiplicity) contributes at least $\epsilon \ell / 2$ to $\deg R_\phi$. Let $m$ be the number of such fibers. Since 
    \[-2\chi(X)= -2\ell\chi(X_0) + \deg R_\phi\]
    and $\deg R_\phi \geqslant m \epsilon \ell /2$ we conclude that
    \[m \leqslant \frac{4}{\epsilon} \left(\frac{-
    \chi(X)}{\ell}+\chi(X_0)\right).\]
    So the number of $\epsilon$-ramified points is at least $b-m \geqslant N$, as claimed.
\end{proof}

It is convenient to state the results that follow in terms of the quantity $N$ directly, and use Lemma \ref{lem:mystery} or similar results to later connect $N$ to more commonly used invariants of coverings. The parameters $\epsilon, \ell$ cannot be entirely arbitrary, and we want to choose them so that $\epsilon$-ramified branch points exist, but are such that at least half of the fiber above them is unramified. For this we introduce the following numerical assumption.

\begin{assumption}\label{assmpt:lepsilon}
    The integer $\ell$ and the real number $\epsilon$ satisfy 
    \[\lceil (1-\epsilon)\ell \rceil \geqslant \ell/2 \text{ and } \epsilon \geqslant 2/\ell.\]
    If $\ell \geqslant 6$ assume in addition $\lceil (1-\epsilon)\ell \rceil\geqslant \ell/2  + 1$.
\end{assumption}
\begin{rem}
    Note that Assumption \ref{assmpt:lepsilon} automatically implies $\ell \geqslant 4$. On the other hand,  taking $\epsilon=1/2$ for $\ell=4,5$ and $\epsilon =1/3$ for $\ell \geqslant 6$ shows that a suitable $\epsilon$ can be chosen to satisfy \Cref{assmpt:lepsilon} for every $\ell\geqslant 4$.
\end{rem}

We are now ready to prove the main geometric result of this section.

\begin{thm}\label{thm:S-HIT_geom}  Suppose we are in Setup \ref{main-setup}, and Assumption \ref{assmpt:lepsilon} holds. Suppose $\phi$ admits at least $N$ $\epsilon$-ramified branch points. Let $\psi_m:Y_m \to Y_{m-1}$ denote the natural projection of degree $\ell-m+1$. Then  the Riemann--Hurwitz contribution is at least:
\[\deg R_{\psi_m}=2\left(-\chi(Y_m)+(\ell-m+1)\chi(Y_{m-1})\right) \geqslant  N \lceil(1-\epsilon)\ell\rceil^{\uline{m-1}}.\] If, moreover, $N$ is chosen so that $N \leqslant 2(\ell-2)$,  then for every $m \neq 1, \ell-1$ we have \[ \gon (Y_m) \geqslant \frac{N}{2(\ell-m)}\lceil(1-\epsilon)\ell\rceil^{\uline{m-1}},\text{ and }\,  \gon (X_m) \geqslant \frac{N}{2m(\ell-m)} {{\lceil(1-\epsilon)\ell\rceil}\choose {m-1}}.\] 
\end{thm}

\begin{proof}
    Let $\chi_m$ denote the Euler characteristic of $Y_m$. Our main goal is to estimate $\gon(X_m)$ using bounds on $\gon(Y_m)$. Since $X_m \simeq X_{\ell - m}$, in the following we assume $m \leqslant \ell/2$. We need to compute the Riemann--Hurwitz contribution of $\psi_m$. To do so, we interpret points on the curves $Y_m$ as ordered $m$-tuples of distinct points in fibers of $\phi$. Choose an $\eps$-ramified branch point $p \in X_0$, so that the fiber $\phi^{-1}(p)$ has at least $\lceil(1-\epsilon)\ell\rceil$ unramified points. Let $\gamma\in G$ be a branch cycle over $p$, which we view as acting on $\ell$ points. Since $\gamma$ fixes at least $\lceil(1-\epsilon)\ell\rceil \geqslant \ell/2$ points, the covering $Y_{m-1} \to X_0$ has at least $\lceil(1-\varepsilon)\ell\rceil^{\uline{m-1}}$ unramified points above $p$. Each of those points, when thought of as an ordered $m$-tuple, can be completed to a non-fixed $m$-tuple, and so by Riemann-Hurwitz  we can write
    \[-2\chi_m\geqslant -2(\ell-m+1)\chi_{m-1} + N \lceil(1-\epsilon)\ell\rceil^{\uline{m-1}},\]
    and therefore 
    \begin{equation}\label{equ:degRpsi}-2\left(\chi_m-(\ell-m+1)\chi_{m-1}\right) \geqslant N \lceil(1-\epsilon)\ell\rceil^{\uline{m-1}},\end{equation}
    as claimed.

    We now prove that $\gon(Y_m) \geqslant \frac{N}{2(\ell-m)}\lceil(1-\epsilon)\ell\rceil^{\uline{m-1}}$ by induction on $m \geqslant 2$.
    If $m=2$ we are asked to prove \[\gon(Y_2) \geqslant \frac{N}{2(\ell-2)}\lceil(1-\varepsilon)\ell\rceil.\] Note that by our assumption $N\leqslant 2(\ell-2)$ and by \Cref{assmpt:lepsilon}, the right hand side is smaller than $\ell-1$. Applying \Cref{Castelnuovo} and using our calculation \eqref{equ:degRpsi}, 
    
    we get 
    \[\gon(Y_2) \geqslant \min\left\{\ell-1,\frac{(\ell-1)
    \chi_1-\chi_2}{\ell-2}\right\}\geqslant \frac{N}{2(\ell-2)}\lceil(1-\epsilon)\ell\rceil,\] as desired.
    For higher $m$ we proceed similarly. First we note that $\deg \psi_m=\ell-m+1$ and so   the induction hypothesis implies: 
    \[\deg \psi_m\cdot  \gon(Y_{m-1}) \geqslant (\ell-m+1)  \frac{N}{2(\ell-m+1)}\lceil(1-\epsilon)\ell\rceil^{\uline{m-2}} \geqslant \frac{N}{2(\ell-m)}\lceil(1-\epsilon)\ell\rceil^{\uline{m-1}}.\] Applying \Cref{Castelnuovo} to the indecomposable covering $\psi_m$,  we get
    \[\gon(Y_m)\geqslant 
    \min\left\{\deg \psi_m\cdot  \gon(Y_{m-1}), \frac{(\ell-m+1)\chi_m-\chi_{m-1}}{\ell-m}\right\}\geqslant 
    \frac{N}{2(\ell-m)}\lceil(1-\epsilon)\ell\rceil^{\uline{m-1}}\]
as claimed.    

Finally, by \Cref{lem:gon-and-airr}, 
\[\gon(X_m)\geqslant \frac{\gon (Y_m)}{m!} \geqslant \frac{1}{m!} \frac{N}{2(\ell-m)} \lceil(1-\epsilon)\ell\rceil^{\uline{m-1}}= \frac{N}{2m(\ell-m)} {{\lceil(1-\epsilon)\ell\rceil}\choose {m-1}}. \qedhere
\]
\end{proof}

The functions which are lower bounds for gonalities $\gon (X_m)=\gon (X_{\ell-m})$ grow quickly with $m$, but start to decrease when $m$ is close to $\ell/2$. Since the lower bounds are always large near $\ell/2$, we use the following lemma for proving uniform lower bounds on gonalities.
\begin{lem}[Useful calculus]\label{lem:calculus}
    Suppose Assumption \ref{assmpt:lepsilon} holds, and $N$ is a positive integer. Then the function \[f(m)=\frac{N}{2m(\ell-m)} {{\lceil(1-\epsilon)\ell\rceil}\choose {m-1}}\]
    of an integer argument $m$ is unimodal on the interval $\left[2, \lfloor \ell/2 \rfloor \right]$:  it  increases up to some point in the interval, and then decreases.
\end{lem}
\begin{proof}
We will show that the function $f(m)/f(m-1) - 1$ is positive at $2$ and changes sign at most once on the interval $[2, \lfloor \ell/2 \rfloor]$. We compute
    \[\frac{f(m)}{f(m-1)}-1 = \frac{(\ell-m+1)(\lceil(1-\epsilon)\ell\rceil - m +2)}{m(\ell-m)} - 1.\] Removing the positive denominator, we need to show that the following function in $m$ is positive at $2$ and changes sign at most once on the interval: \[g(m)=(\ell-m+1)(\lceil(1-\epsilon)\ell\rceil-m+2)-m(\ell-m).\]

    The function $g(m)$ is a quadratic polynomial in $m$. The vertex of the parabola is located at the point $(2\ell+\lceil(1-\epsilon)\ell\rceil+3)/4>\ell/2$, and so, indeed, the function can change sign at most once on the interval $[2, \ell/2]$. Direct substitution also gives $g(2)>0$.  
\end{proof}

To analyze irreducibility of fibers of a covering $\phi: X \to X_0$ we need to control the invariant $\min \delta(X_m)$, and \Cref{thm:S-HIT_geom} gives us control over $\gon(X_m)$. These two invariants might differ by a factor of two (\Cref{lem:gon-and-airr}). Since $\gon (X_m)$ grows quickly with $m$, it is enough to understand $\min \delta (X_2)$ to understand $\min_m \min \delta(X_m)$. This is done in the next lemma.

\begin{lem}\label{lem:min-delta-X2}
Suppose we are in Setup \ref{main-setup} and Assumption \ref{assmpt:lepsilon} holds. Suppose $\phi$ has at least $N$ $\epsilon$-ramified points for some $N\leqslant 2(\ell-2)$. Then \[\min \delta (X_2) \geqslant \frac{N (1-\epsilon)}{4}.\]
\end{lem}
\begin{proof}
    By \Cref{lem:gon-and-airr}, it is enough to prove that $\min \delta (Y_2) \geqslant N (1-\epsilon)/2$. Note that for $
    \ell\leqslant 5$, one has $N(1-\eps)/4\leqslant (1/2)(\ell-2)(1-2/\ell)<1$ by Assumption \ref{assmpt:lepsilon},  
     
    so we may assume $\ell \geqslant 6$. For the same reason, we can assume $N(1-\epsilon)>4.$ Suppose $d=\min \delta (Y_2) < N(1-\epsilon)/2$. Then part $(2)$ of \Cref{prop:arithm-gonality} shows that there is an integer $e$ and a degree-$e$ covering $\psi: Y_2 \to Z$  such that $\min \delta(Z)=d/e$ and
        $-\chi(Z) \leqslant \frac{3}{4} \left(d/e\right)^2 + d/e-1$. We analyze this covering depending on the value of $e$, using the inequality \begin{equation}\label{equ:chiN}-\chi(Y_2)+(\ell-1)\chi(X) \geqslant  \frac{N(1-\epsilon)}{2}\ell\end{equation}
    of Theorem \ref{thm:S-HIT_geom}. 

    Suppose $e=1$, and so  $Y_2=Z$. Then the last two inequalities above give: \[\frac{3}{4} d^2 + d  - 1 \geqslant -
    \chi(Y_2) \geqslant -(\ell - 1) + \frac{N(1-\epsilon)}{2}\ell.\]
    Substituting $d=1$ gives an apparent contradiction, so we can assume that $d\geqslant 2$. Since $d< N(1-\epsilon)/2$ it suffices to check that the weaker inequality
    \[\frac{3}{4}d^2+d-1 \geqslant -(\ell-1)+d\ell\] cannot hold. Indeed, the quadratic polynomial $(3/4) d^2 +(1-\ell)d + (\ell-2)$ takes negative values when $d=2$ and $d=\ell-4$. But the largest possible value for $d$ is $\lfloor N(1-\eps)/2 \rfloor$, and  \[\lfloor N(1-\epsilon)/2\rfloor \leqslant \lfloor (1-\epsilon)(\ell-2) \rfloor\leqslant\left\lfloor(\ell-2)\left(1-\frac{2}{\ell}\right)\right\rfloor=\left\lfloor\ell-4+
    \frac{4}{\ell}\right\rfloor=\ell-4.\] Therefore the polynomial $(3/4) d^2 +(1-\ell)d + (\ell-2)$ takes a negative value for all $d$ in the interval $[2, \lfloor N(1-\epsilon)/2\rfloor]$, giving the desired contradiction.

    Suppose $e=d$. Then $\psi$ is a degree $d$ covering from $Y_2$ to a curve of genus $0$ or $1$. Since $d<\ell -1$, the map $\psi$ cannot factor through the map $\eta: Y_2 \to X$. Since the latter is indecomposable, we can apply the Castelnuovo--Severi inequality to the pair of maps $\psi, \eta$ to get
    \[-\chi(Y_2) \leqslant -(\ell-1)\chi(X) + d(\ell-1). \]
    Together with \eqref{equ:chiN} this gives:
    \[\frac{N (1-\epsilon)}{2} \ell \leqslant d(\ell - 1),\]
    which implies $d \geqslant N(1-\epsilon)/2$, violating the assumption.

    Therefore we can assume $2 \leqslant e \leqslant d/2$. Note that in this case $\ell \geqslant 7$, since otherwise $d< N(1-\epsilon)/2\leqslant (\ell-2)(1-\epsilon)<4$, and thus $d$ does not factor. In this case we can again apply Castelnuovo--Severi to the coverings $\psi$ and $\eta: Y_2 \to X$ which gives 
    \[-\chi(Y_2)\leqslant -e\chi(Z) -(\ell-1)\chi(X) + e(\ell-1).\]
    Rearranging and using \eqref{equ:chiN} gives
    \[\frac{N(1-\epsilon)}{2}\ell \leqslant (\ell-1)\chi(X)-\chi(Y_2)\leqslant e(\ell-1)-e\chi(Z)\leqslant e(\ell-1)+e\left(\frac{3}{4}(d/e)^2+(d/e)-1\right), \]
    and collecting $d$-dependent terms on one side gives
    \[\frac{3d^2}{4e}+d\geqslant \frac{N(1-\epsilon)}{2}\ell -e(\ell-2).\]
    The left hand side is monotone in $d$, and so it is enough to show that the inequality cannot hold for $d=\left\lfloor N(1-\epsilon)/2\right\rfloor \leqslant \ell-4.$
    Replacing $N(1-\epsilon)/2$ with $d$ on the right gives
    \[ \frac{3d^2}{4e}-(\ell-1)d+e(\ell-2)
\geqslant 0.\]

    The expression $3d^2/4e+e(\ell-2)$ is a monotone function of $e\in [2, d/2]$. Therefore we need to check the feasibility of the inequality for $e=2$ and $e=d/2$. Substituting $e=d/2$ gives 
   \[(3/2)d-(\ell-1)d+(\ell-2)d/2 \geqslant 0,\]
which does not hold for $\ell \geqslant 6$. Substituting $e=2$ gives
\[\frac{3d^2}{8}-(\ell-1)d+2(\ell-2)\geqslant 0, \]
which does not hold for $2\leqslant d\leqslant \ell-4$, as can be seen by verifying the inequality at the ends of the interval (using the inequality $\ell \geqslant 7$). This is a contradiction.
\end{proof}

We can now prove the main result of this section.

\begin{thm}\label{thm:S-HIT}    
    Suppose we are in Setup \ref{main-setup} with $\ell \geqslant 6$ and Assumption \ref{assmpt:lepsilon} holds.  Suppose $\phi$ has at least $N$ $\epsilon$-ramified points for some $N\leqslant 2(\ell-2)$. Then for all but finitely many points $x \in X_0$ of degree $d < \frac{N(1-\epsilon)}{4}$  the fiber $\phi^{-1}(x)$ viewed (as a scheme) over $k(x)$ is either irreducible or splits into two closed points of degrees $1$ and $\ell-1$.
 \end{thm}
 \begin{proof}
     
     As in the beginning of Section \ref{sec:resolvent},  the assertion will follow from the inequality \[\min \delta (X_m) \geqslant \frac{N (1-\epsilon)}{4}\] for $2 \leqslant m \leqslant \lfloor \ell/2 \rfloor$. The inequality holds for $m=2$ by \Cref{lem:min-delta-X2}. For $3 \leqslant m \leqslant \lfloor \ell/2 \rfloor$, we can use \Cref{prop:arithm-gonality} and Theorem \ref{thm:S-HIT_geom} to write 
     \[\min \delta (X_m) \geqslant \gon(X_m)/2 \geqslant \frac{N}{4m(\ell-m)}  {{\lceil(1-\epsilon)\ell\rceil}\choose {m-1}}.\]
     By \Cref{lem:calculus} it is enough to check the inequality \[\frac{N}{4m(\ell-m)}  {{\lceil(1-\epsilon)\ell\rceil}\choose {m-1}} \geqslant \frac{N(1-\epsilon)}{4}\] for $m=3$ and for $m=\lfloor \ell/2 \rfloor$.
     For $m=3$ we need to verify
     \[{{\lceil(1-\epsilon)\ell\rceil}\choose {2}} \geqslant (1-\epsilon)3(\ell-3).\]
     Since $\lceil(1-\epsilon)\ell\rceil \geqslant \ell/2$ it is enough to check that
     \[\ell(\ell-2) \geqslant 12(1-\epsilon)(\ell-3),  \]
     which holds for $\ell \geqslant 6$ since $\eps\geqslant 2/\ell$. 
     
     For $m=\lfloor \ell/2 \rfloor$, we get 
     \[{{\lceil(1-\epsilon)\ell\rceil} \choose {\lfloor \ell/2 \rfloor -1}} \geqslant (1-\epsilon) \lceil \frac{\ell}{2}\rceil \lfloor \frac{\ell}{2}\rfloor.\]
     Since $2 \leqslant \lfloor \ell/2 \rfloor  -1 \leqslant \lceil(1-\epsilon)\ell\rceil - 2$, one has ${\lceil(1-\eps)\ell\rceil \choose \lfloor \ell/2 \rfloor -1 } \geqslant {\lceil(1-\eps)\ell\rceil \choose 2}$ and $\lceil\ell/2\rceil\lfloor \ell/2\rfloor\leqslant \ell^2/4$, so that is enough to show that 
     \[{\lceil(1-\epsilon)\ell\rceil \choose 2} \geqslant (1-\epsilon)\frac{\ell^2}{4},\]
     which follows from 
     \[\lceil(1-\epsilon)\ell\rceil \geqslant \frac{\ell}{2}+1.\]
     The last inequality holds by our assumptions. 
 \end{proof}
We now collect a few useful corollaries of Theorem \ref{thm:S-HIT}.
\begin{cor}\label{cor:S-HIT-no-N}
    Suppose $\phi:X \to X_0$ is a degree $\ell\geqslant 6$ covering of curves with monodromy $A_\ell$ or $S_\ell$ and  at least $b$ branch points. Let $\epsilon \in [2/\ell, 1/2-1/\ell]$ be any constant. Then for all but finitely many points $x \in X_0$ of degree \[d< \frac{1-\epsilon}{4} \min \left(2(\ell-2), b-\frac{4}{\epsilon}\left(\chi(X_0)-\frac{\chi(X)}{\ell}\right)\right)\]  the fiber $\phi^{-1}(x)$ viewed (as a scheme) over $k(x)$ is either irreducible or splits into two closed points of degrees $1$ and $\ell-1$.
\end{cor}
\begin{proof}
    This follows from combining  \Cref{thm:S-HIT} with the estimate on $N$ from \Cref{lem:mystery}. The requirement $N\leqslant 2(\ell-2)$ of \Cref{thm:S-HIT} necessitates the introduction of the $\min$ function in the inequality.
\end{proof}
Specializing even further, to high-degree rational functions on a fixed curve $X$, we get:
\begin{cor}\label{cor:irred-rat-funct}
    Suppose $\phi:X \to\mP^1_k$ is a degree $\ell\geqslant \max(6, 2g)$ rational function on a curve $X$ of genus $g$ with monodromy $A_\ell$ or $S_\ell$ and  at least $b$ branch points. Then for all but finitely many points $x \in X_0$ of degree  \[d<\min\{1/3(\ell-2), b/6-3\}\] the fiber $\phi^{-1}(x)$ viewed (as a scheme) over $k(x)$ is either irreducible or splits into two closed points of degrees $1$ and $\ell-1$.
\end{cor}
\begin{proof}
    This follows from \Cref{cor:S-HIT-no-N} with $\epsilon=1/3$ using the bound $-\chi(X)/\ell<1/2$. 
\end{proof}

\begin{rem}\label{b/7-abstract}
    Consider a rational function $\phi: X \to \mP^1_k$ of degree $\ell \gg g$. Then the Riemann--Hurwitz formula gives $2g-2=-2\ell + \deg R_\phi \geqslant -2\ell + b$, and so $b < 2\ell + 2g$. Applying \Cref{cor:S-HIT-no-N} with $\epsilon=1/3$ and noting that $b/7<(2/7) \ell + 2g<\ell/3 -2$, shows that the conclusion of \Cref{cor:irred-rat-funct} holds for $d<b/7-2$ (by using $g/\ell\approx 0$ in the formula of \Cref{cor:S-HIT-no-N}).
\end{rem}
For polynomials --- the example with which we began the paper --- one can obtain a more precise result, by taking into account the ramified cycle at infinity. Moreover, the monodromy assumption in this case is automatically satisfied.
\begin{cor}\label{cor:hilb-for-poly}
    Suppose $\phi\in k[x]$ is an indecomposable polynomial with $b \geqslant 12$ branch points in $\mathbb{A}^1$. Then for all but finitely many $t \in \overline{k}$ of degree \[d< b/4 - \sqrt{b/2} + 1/2\] the polynomial $\phi(x)-t$ is either irreducible or factors into a linear factor and a degree $\deg \phi-1$ factor over the field $k(t)$.  
\end{cor}
\begin{rem}
    Since $b/4-\sqrt{b/2}+1/2\geqslant b/6-1$ this implies \Cref{prop:intro}.
\end{rem}
\begin{proof}
    Since $\phi$ is a polynomial and $b\geqslant 12$, the monodromy group of $\phi$ is either symmetric or alternating by \cite[Theorem 1.0.2]{GS}, 
     
    or Remark \ref{rem:Bochert} below (which does not rely on the classification of finite simple groups), and so we are in position to apply \Cref{thm:S-HIT}. To estimate $N$ we proceed similarly to \Cref{lem:mystery}, but take into account the totally ramified point at infinity. Let $\ell$ denote the degree of $\phi$. Let $m$ be the number of finite branch points of $\phi$ which are \emph{not} $\epsilon$-ramified. The Riemann--Hurwitz contribution over each of these points is at least $\eps\ell/2$ and that over the point at infinity is $\ell-1$. Thus the Riemann--Hurwitz formula yields:
    \[-2 = 2\chi(\mP^1) = \ell\chi(\mP^1)+\deg R_\phi\geqslant  -2\ell + (\ell-1) + m\epsilon \ell/2,\]
    giving $m \leqslant 2/\epsilon$. Therefore $\phi$ has at least $N=b-2/\epsilon$ $\epsilon$-ramified branch points.

    Take $\epsilon=(2/b)^{1/2}$. Then, noting that $\ell \geqslant b+1\geqslant 13$ by the Riemann--Hurwitz formula, the numerology of \Cref{assmpt:lepsilon} is satisfied. Applying \Cref{thm:S-HIT} we see that the conclusion holds for points of degree at most
    \[\frac{N(1-\epsilon)}{4}=\frac{(b-\sqrt{2b})(1-\sqrt{2/b})}{4}=\frac{b-2\sqrt{2b}+2}{4}.\]
\end{proof}

\begin{rem}\label{rem:Bochert}
In the setup of  \Cref{cor:hilb-for-poly}, we now explain how to prove that $G=\Mon(f)$ is $A_\ell$ or $S_\ell$ by elementary means. 
 
 Since $f$ is an indecomposable polynomial, $G$ is primitive and contains an $\ell$-cycle. Hence, either $G$ is doubly-transitive almost-simple or $f$ is the composition of $X^\ell$ or $T_\ell$ with linear polynomials, by classical theorems of Burnside, Chisini, Schur, and Ritt; see \cite[Theorem 2.1]{konig2024reducible}\footnote{Note that  \cite[Theorem 2.1]{konig2024reducible} indeed assumes that $f$ is indecomposable.}. Since $b\geqslant 12$, we deduce $G$ is doubly-transitive and hence of even order (e.g.\ since it contains an element swapping to $1,2$) and so contains an involution. We may therefore apply classical bounds on the minimal degree $m$ of $G$ due to Bochert \cite{bochert1892ueber} and their modern exposition \cite{schomburg2022bochert}; 
 
 Recall that the minimal degree $m$ of $G$ is the minimal number of elements moved by a nontrivial element of $G$. 
 If $G\neq A_\ell,S_\ell$, one has $m\geqslant 4$ by Jordan's theorem \cite[Theorem 3.3E]{dixon1996permutation}. \cite[Theorem 3.1]{schomburg2022bochert}\footnote{The inequality we use is obtained at the end of the proof of Theorem 3.1 and does not require the assumption $\ell\geqslant 38$.} then implies  $\ell\leqslant 4m+6/(m-3)$. 
 On the other hand, the index over each of the $b\geqslant 12$ finite branch points is at least $m/2$. 
 Hence, the Riemann--Hurwitz contribution $\deg R_f=2\ell-2$ is at least $(\ell-1)+b\cdot m/2\geqslant (\ell-1)+6m$, so that $\ell-1\geqslant 6m$. In total we get $4m+6/(m-3)\geqslant \ell\geqslant 6m+1$, contradicting that $m\geqslant 4$.
 
\end{rem}
\section{Highly branched covers and Vojta's inequality}\label{sec:Vojta}

The results of the previous section work well in the case of large $\ell$. However, it is also interesting to study the case of fixed degree $\ell$ maps with a growing number of ramification points; for instance, it is interesting to study the case of a large genus $g$ hyperelliptic curve $X$ and its degree two map $X \to X_0=\mP^1$. The case of hyperelliptic curves is completely analyzed in the following theorem of Vojta, obtained using analytic methods in his seminal work \cite{vojta1992generalization}. Vojta's theorem, and its generalizations, are proved in the literature under the assumption that $k$ is a number field (rather than a general finitely generated field). However, by a specialization argument that we present in \Cref{rem:number-fields}, these results hold more generally over finitely generated fields of characteristic zero.  

\begin{prop}[{\cite[Corollary 0.3]{vojta1992generalization}}]\label{vojta}
    Suppose $k$ is a number field, and $\phi: X \to \mP^1_k$ is a degree $\ell$ covering of curves. If the Euler characteristic of $X$ satisfies 
    \[-\chi(X) > \ell(d-1).\]
    then the set \[\{P \in |X|:\ \deg P = \deg \phi(P)=d\}\] is finite.
\end{prop}

This implies that for the hyperelliptic map $\phi: X \to \mP^1$ all but finitely many fibers $\phi^{-1}(P)$ are irreducible for $\deg P < (g+1)/2$. To work with more general coverings of curves, we will use the following generalization of Vojta's theorem due to Song and Tucker.
\begin{prop}[Proposition 2.3 of \cite{song2001arithmetic}+\Cref{rem:number-fields}]\label{song-tucker}
Suppose $k$ is a finitely generated field of characteristic $0$, and $\phi:X \to X_0$ is a degree $\ell$ covering of curves. If 
\[\ell \chi(X_0) - \chi(X)>d\ell,\]
then the set \[\{P \in |X|:\ \deg P = \deg \phi(P)=d\}\] is finite.
\end{prop}
\begin{rem}
   The bound in \Cref{vojta} is sharp: the examples are provided by $(\ell, d)$ curves $X \subset \mP^1_k \times E$ for an elliptic curve $E$ (more specifically, by smooth curves in the algebraic equivalence class of $dE+\ell \mP^1_k$: a combination of vertical and horizontal lines). Let $\pi_E :X \to E$, $\pi_{\mP^1}:E \to \mP^1$ denote the projections. By the adjunction formula, the genus of such a curve coincides with the bound $\ell(d-1)+1$ from \Cref{vojta}. By the potential Hilbert's irreducibility theorem for elliptic curves \cite[Theorem 2.5]{tucker2002irreducibility}, for some field extension $L/k$ the set $\{p\in E(L)\colon \deg \pi_E^{-1}(p)=\deg \pi_{\mP^1}(\pi_E^{-1}(p))=d\}$ is infinite so that the map $\pi_{\mP^1}$ does not satisfy the conclusion of \Cref{vojta}. A similar construction works for some $(\ell, d)$ and curves on $X_0\times E$, showing the  sharpness of \Cref{song-tucker}\footnote{Note that \cite[Corollary 2.1]{song2001arithmetic} contains a typo in the statement of Vojta's theorem: the inequality sign in the theorem has to be strict.}. 
\end{rem}
 
\begin{rem}\label{rem:song-tucker-geom-irr}
    \Cref{song-tucker} is stated in \cite{song2001arithmetic}, but we believe that geometric irreducibility is tacitly assumed in some proofs. Nevertheless, it is easy to see the statement of \Cref{song-tucker} for geometrically irreducible curves implies the statement for all curves (using \Cref{prop:curve-trivialities}). Namely, without loss of generality we can assume that $X_0$ is geometrically irreducible and let $L$ be the field of constants of  $X$. Applying the theorem to the covering $X/L \to (X_0)_L$ of degree $\ell/[L:k]$, and using \Cref{prop:curve-trivialities}, gives the conclusion for points of degree
    \[d_L<\frac{ \chi((X_0)_L)\cdot \ell/[L:k]-\chi(X/L)}{\ell/[L:k]}=\frac{\chi((X_0)_L)\cdot \ell/[L:k]-\chi(X)/[L:k]}{\ell/[L:k]}=\frac{\ell\chi(X_0)-\chi(X)}{\ell}.\]
    Since degree $d$ points on $X$ correspond to points of degree $d_L=d/[L:k]$  on $X/L$ the result follows.
\end{rem}

A direct application of \Cref{song-tucker} gives the following result.
\begin{prop}\label{thm:large-branch-main}
    Suppose $\phi:X \to X_0$ is a covering of curves over $k$ with monodromy group $G$ and $b$ branch points.  Let $n_{\max{}}$ denote the maximal index of a subgroup $H\sub G$, which itself is maximal among the subgroups with trivial core (in particular $n_{\max}\leqslant |G|$). Then for all but finitely many degree $d<b/(2n_{\max{}})$ points $P$ of $X_0$, the  image $\Im \varphi_P$ 
    contains a nontrivial normal subgroup $1 \neq N \lhd G$. 
\end{prop}
\begin{proof}
    Let $H\subset G$ be a subgroup maximal among subgroups with trivial core. It suffices to show that there are at most finitely many points $P$ of degree $d<b/2n_{\max}$ for which the Galois group of the fiber is contained in $H$. Let $\tilde{X}\to X$ be the Galois closure of $\phi$, and let $\phi_H:\tilde{X}/H \to X_0$ denote the covering corresponding to the subgroup $H$. Note that, by definition, $\deg \phi_H \leqslant n_{\max}.$ We need to show that the set $\{P \in |X_H|\colon \deg P =\deg \phi_H(P)=d\}$ is finite. For this we directly apply \Cref{song-tucker}. Since $H$ has no core, every branch point of $\phi$ is also a branch point of $\phi_H$. By Riemann--Hurwitz, we can write
    \[2\chi(X_H)=2\deg(\phi_H)\chi(X_0)-\deg R_{\phi_H}. \]
    Since $\deg R_{\phi_H}\geqslant b$ as above,  \Cref{song-tucker} implies the assertion  for points of degree at most \[\frac{\deg \phi_H \cdot \chi(X_0)-\chi(X_H)}{\deg \phi_H}=\frac{\deg R_{\phi_H}}{2\deg \phi_H}\geqslant\frac{b}{2n_{\max}}, \]
     as needed.
\end{proof}
In a similar way, we obtain the following result on the irreducibility of fibers.

\begin{thm}\label{thm:vojta-new}
    Suppose $\phi:X \to X_0$ is a covering of degree $\ell$ and monodromy $A_\ell$ or $S_\ell$. 
    Then the fiber of $\phi$ is irreducible over all but finitely many points of degree $d<b/(2\ell).$  
\end{thm}
\begin{proof}
    As before, we need to estimate $\deg R_{\phi_m}$. Note that each of the branch points of $\phi$ is also a branch point for $\phi_m$. Moreover, by \cite[Lemma 2.0.19]{GS}, the index of any branch cycle in the action on $m$-sets is at least $\binom{\ell-2}{m-1}$. Therefore $\deg R_{\phi_m} \geqslant b\binom{\ell-2}{m-1}$ and applying \Cref{song-tucker} shows that the result holds for points of degree at most
    \[\frac{\deg R_{\phi_m}}{2\deg \phi_m} \geqslant \frac{b\binom{\ell-2}{m-1}}{2\binom{\ell}{m}}=\frac{bm(\ell-m)}{2\ell(\ell-1)}\geqslant \frac{b}{2\ell},\]
    as claimed. 
\end{proof}

Finally, \Cref{song-tucker} combined with the Burness--Guralnick index bounds can be used to obtain the following variant of \Cref{cor:BG}.

\begin{cor}\label{cor:BGST}
    Let $\phi:X\to X_0$ be an indecomposable cover of degree $n$,  monodromy group $G$, and $b$ branch points. Suppose it is \emph{not} the case that $A_\ell^t\leqslant G\leqslant S_\ell\wr S_t$, $\ell\geqslant 5$, $t\geqslant 1$ is of product type with respect to the $S_\ell$-action  on $m$-sets for   $1\leqslant m\leqslant \ell/2$.  Then the set 
    \[\{P \in |X|:\ \deg P = \deg \phi(P)< 3b/28\}\]
    is finite.
    
\end{cor}
\begin{proof}
Assume that the action of $G$ is not one of the above product-type actions. Then  $\ind(G)\geqslant 3n/14$ by \cite[Theorem 7]{BG}. By \Cref{song-tucker}, the conclusion holds for points of degree $d$ less than
\[\frac{\deg R_\phi}{2n}\geqslant \frac{b\cdot \ind(G)}{2n}\geqslant 3b/28 , \]
as claimed. 
\end{proof}

\section{Preservation of Galois groups}\label{sec:Galois-groups}
In this section, we further consider the entire Galois group of fibers of maps $\phi:X\to X_0$  when $\Mon(\phi)=A_\ell$ or $S_\ell$. Recall that $\varphi_P:\Gal_{k(P)}\to S_\ell$ denotes the Galois action on the fiber above a point $P\in |X_0|_d$. Our main result is the following theorem.

\begin{thm}\label{thm:other-actions}
    Let $\phi:X\to X_0$ be a map of degree $\ell\geqslant 6$, with  monodromy group $G=A_\ell$ or $S_\ell$, and $b$ branch points. Suppose \Cref{assmpt:lepsilon} holds. Suppose at least $N \leqslant 2(\ell-2)$ branch points of $\phi$ are $\eps$-ramified.  
    Then the following hold:
    
    \begin{enumerate}
        \item\label{S-l-Galois-case} $\Im \varphi_P\cong A_\ell$ or $S_\ell$ for all but finitely many points $P$ of degree $d< \min\{N(1-\eps)/4,\lfloor b/16\rfloor + 1\}$ outside $\phi(|X|_d)$;
        
        \item\label{S-l-Galois-fixed-case} If we assume in addition $\ell\geqslant 7,$ $\eps\leqslant 1/2- 3/(2\ell),$ and $N\leqslant 2(\ell-4)$, then $\Im\ \varphi_P=A_{\ell-1}$ or $S_{\ell-1}$ for all but finitely many points $P \in \phi(|X|_d)$ of  degree $d<\min\{N(1-\eps)/4,\lfloor b/16\rfloor + 1\}$.         
    \end{enumerate}
 
\end{thm}
Note that  if $\Mon_k(\phi) = S_\ell$, the theorem permits $\Im\varphi_P=A_\ell$, 
which is not the full image, for infinitely many degree $d$  points $P$. 
\begin{proof}[Proof of Theorem \ref{thm:other-actions}]  

We first prove part \eqref{S-l-Galois-case}. Suppose there are infinitely many degree $d$ points $P$ of $X_0$ such that $\Im\varphi_P$ is a given subgroup $H\leqslant S_\ell$ and $P \not \in \phi(|X|_d)$. Assuming to the contrary that $H\notin \{A_\ell,S_\ell\}$, we embed $H$ in a subgroup $M\neq A_\ell,S_\ell$ which is maximal among subgroups not equal to $A_\ell$ or $S_\ell$. In particular, $M$ is either a maximal subgroup of $S_\ell$ or a maximal subgroup of $A_\ell$ that is not contained in any other maximal subgroup of $G$. 
     
    Let $\Omega\coloneqq G/M$ and $n\coloneqq|\Omega|$, so that $n\geqslant \ell$. 
    Consider the curve $X_M\coloneqq\tilde{X}/M$, where $\tilde{X} \to X_0$ is the Galois closure of $\phi$, and let $\psi:X_M\to X_0$ be the natural projection.

    Case 1: $M$ is intransitive. Thus $M$ is the stabilizer  $G_k=(S_k\times S_{\ell-k})\cap G$ of a set of cardinality $k\geqslant 1$. 
     
    Since the points $P$ are not in $\phi(|X|_d)$, we may assume $2\leqslant k\leqslant \ell/2$, and since $N\leqslant 2\ell-2$,  Theorem \ref{thm:S-HIT} implies $d\geqslant \min\delta(X_M)\geqslant N(1-\eps)/4$.

    Case 2: $M$ is transitive.  Assume at first that $M$ is maximal in $G$. Then the action of $G$  on $\Omega$ is primitive and is not one of the product-type actions in Corollary \ref{cor:BG}, so that the corollary implies $d\geqslant \min\delta(X_M)\geqslant \min\{n/2,b/16\}$. Since $N<2\ell$, and $n\geqslant \ell$,  we have $n/2\geqslant \ell/2>N(1-\eps)/4$ and hence $\min\delta(X_M)\geqslant \min\{N(1-\eps)/4,b/16\}$ as needed. 
    
    Assume next that $M$ is not maximal in $G$, in which case  $G=S_\ell$ and $M$ is a maximal subgroup of $A_\ell$ contained properly in no subgroup of $S_\ell$ other than $A_\ell$. In this case $\gon(X_M)\geqslant \min\{n/2,\deg R_{\psi}/(2(n-1))\}$ by  \Cref{rem:Cast-generalized}. Write $b=b_o+b_e$, where $b_e$ (resp.\ $b_o$) is the number of  branch points $p$ with an even (resp.\ odd) branch cycle $\gamma_p\in G$. Setting $X_{A_\ell}\coloneqq \tilde X/A_\ell$, we see that the odd branch points are also branch points under  $X_{A_\ell}\to X_0$ and the even ones are not. Since $X_M\to X_0$ factors through $X_{A_\ell}\to X_0$, for each of the $b_o$ odd branch points $p$, all ramification indices above $p$ are divisible by $2$ (and in particular $\psi^{-1}(p)$ has no unramified points). Thus all orbits of $\gamma_p$ are of length $2$ or more, and so we have $\ind_\Omega(\gamma_p)\geqslant n/2$. 
    On the other hand, each of the $b_e$ even branch points $p$ has two preimages in $X_{A_\ell}$ which are branch points of $X_M\to X_{A_\ell}$. The covering $X_M\to X_{A_\ell}$ is an indecomposable covering of monodromy $A_\ell$ and degree $n/2$. The Burness--Guralnick bounds \cite[Theorem 7]{BG} imply that the index of every branch cycle of $X_M \to X_{A_\ell}$ is at least $n/8,$  and hence for our covering $\psi\colon X_M \to X_0$ we conclude that
    $\ind_\Omega(\gamma_p)\geqslant 2\cdot (n/8)=n/4.$ 
     In total, we get $\deg R_\psi\geqslant b_on/2 + b_en/4 \geqslant bn/4$ and hence $\gon(X_M)\geqslant \min\{n/2,bn/(8(n-1))\}$. 
    
     Since $n> \ell>N(1-\eps)/2$ (for groups $M$ of this form), this gives \[\gon(X_M)\geqslant \min\{n/2,bn/(8(n-1))\}>\min\{N(1-\eps)/2,b/8\},\] and hence $d\geqslant \min\delta(X_M)> \min\{N(1-\eps)/4,b/16\}$ by \Pref{prop:arithm-gonality}, yielding \eqref{S-l-Galois-case}.

     Finally, we prove part \eqref{S-l-Galois-fixed-case}  by reducing to part \eqref{S-l-Galois-case}. Let $p$ be an $\eps$-ramified branch point, $\gamma_p\in G$ the corresponding branch cycle,  and $Q\in \phi^{-1}(p)$ an unramified preimage. 
     Let $Y_2=\tilde X/\hat G_2$ be the quotient  by a two-point stabilizer  $\hat G_2\leqslant G$. Then a branch cycle over $Q$ for the map $\phi_2:Y_2\to X$ has the same orbits as $\gamma_p$ except for the removed fixed point $Q$ of $\gamma_p$. In particular, these are branch points of $\phi_2$ that  are $\hat\eps$-ramified, where $\hat \eps=\eps\ell/(\ell-1)$. 
     There are  at least $\lceil(1-\eps)\ell\rceil N$ such branch points in total. 
     Since $\eps\geqslant 2/\ell$, one has $\hat \eps\geqslant 2/(\ell-1)$. Moreover, since $\eps\leqslant 1/2-3/(2\ell)$, one has $\hat \eps\leqslant 1/2-1/(\ell-1)$. 
     As in addition $\ell-1\geqslant 6$, Assumption \ref{assmpt:lepsilon} holds for   the cover $\phi_2:Y_2\to X$ and the values $\hat \eps$ and $N'\coloneqq \min\{2(\ell-3),\lceil(1-\eps)\ell\rceil N\}$ substituted for $\eps,N$, respectively. 
Part \eqref{S-l-Galois-case} of this theorem applied to $\phi_2$ then shows that for all but finitely many points $P\in \phi(|X|_d)\setminus \phi_2(|Y_2|_d)$ of degree $d<\min\{N'(1-\hat \eps)/4,\lfloor N\ell(1-\eps)/16\rfloor+1\}$, one has $\Im\varphi_P=A_{\ell-1}$ or $S_{\ell-1}$. 
     Note that both 
     terms in the minimum are at least $N(1-\eps)/4$ for $\ell\geqslant 7$ since $N\leqslant 2(\ell-4)$ (the inequality $2(\ell-3)(1-\hat{\eps})\geqslant 2(\ell-4)(1-\epsilon)\geqslant N(1-\epsilon)$ is verified through a straightforward, but cumbersome, calculation). 
    
     Finally, note that  $\min\delta(Y_2)\geqslant\min\delta(X_2)\geqslant  N(1-\eps)/4$ by \Cref{lem:min-delta-X2} applied to $\phi$, and hence there are only finitely many  points in $|Y_2|_d$.  
\end{proof}

The proof of Theorem \ref{thm:other-actions} relies on \Cref{cor:BG}, and thus on the classification of finite simple groups. However, we now show that the classification needs only to be applied for groups of small size. To demonstrate this, we prove the main part of \Cref{thm:other-actions}  for $\ell\geqslant 7900$ without relying on the classification.

\begin{thm}\label{thm:ell-actions}
    Let $\phi:X\to X_0$ be a map of degree $\ell\geqslant 7900$, 
     monodromy group $G=A_\ell$ or $S_\ell$, and $b$ branch points.
     Suppose \Cref{assmpt:lepsilon} holds and at least $N\leqslant 2(\ell-2)$ branch points of $\phi$ are $\eps$-ramified. Then $\Im \varphi_P\cong A_\ell$ or $S_\ell$ for all but finitely many points $P$ of degree $d< \min\left\{N(1-\eps)/4,\lfloor b/16\rfloor + 1\right\}$ 
    outside $\phi(|X|_d)$.     
\end{thm}

\begin{proof}

Suppose that there are infinitely many points $P$ with a given Galois image $H\coloneqq \Im\varphi_P\leqslant S_\ell$. If $H$ is intransitive in the standard $S_\ell$-action, 
it is contained in the stabilizer $G_k=G\cap (S_k\times S_{\ell-k})$ of a set of cardinality $k\geqslant 1$. It then follows from Theorem \ref{thm:S-HIT} that $H\leqslant G_1\leqslant G\cap S_{\ell-1}$ fixes a point, so that $P\in \phi(|X|_d)$.  

Henceforth assume that $H$ is transitive. Letting  $\tilde X$ be the Galois closure of $\phi$, we claim that    $X_H\coloneq \tilde X/H$ has gonality  $\gon(X_H)>\min\{(1-\eps)N/2,b/8\}$, so that the conclusion follows from  Proposition  \ref{prop:arithm-gonality}. 
If  $H$ is imprimitive, it is contained in the stabilizer $P_t=G\cap (S_{\ell/t}\wr S_t)$ of a partition with $t$-parts and hence it suffices to bound the gonality of $X_{P_t}$ for $2\leqslant t\leqslant \ell/2$ by Lemma \ref{lem:gon-and-airr}. 
Let $f_t: X_{P_t}\to X_0$ be the natural quotient map, and $n_t$ its degree. (Note that $n_t>\ell>N/2$). 
Then  $\ind(x)\geqslant n_t/4$ for every nontrivial branch cycle $x\in G$, by \cite[Thm.\ 7]{BG}\footnote{The theorem is applied only for the action on partitions, in which case its proof does not apply the classification.}. 
Applying Proposition \ref{Castelnuovo} to $f_t$, we get
$$
\gon(X_{P_t})\geqslant 
\min\left\{n_t\cdot \gon(X_0),\frac{\deg R_{f_t}}{2(n_t-1)}\right\} > 
\min\left\{\frac{N}{2},\frac{1}{2n_t}\cdot \frac{n_tb}{4}\right\}=\min\left\{\frac{N}{2},\frac{b}{8}\right\}.$$

Henceforth assume that $H$ is a proper primitive subgroup of $S_\ell$ other than $A_\ell$. 
Then  $\#H<\exp(4\sqrt{\ell}\log^2\ell)$ for $\ell\geqslant 5000$ by \cite[Theorem 0.1]{babai1981order} and \cite[Theorem A]{pyber}\footnote{For $\ell\geqslant 5000$ the bound in \cite{pyber} is smaller than that in \cite{babai1981order}.}. 
On the other hand $\gon(X_H)\geqslant \gon(\tilde X)/\#H$ by Lemma \ref{lem:gon-and-airr}. 
The gonality of $\tilde X$ is at least the gonality of $Y_{m}=\tilde X/S_{\ell-m}$ for any $m\leqslant \ell/2$. For $m=\ell/2$, since $\eps<1/2$, 
  \Cref{thm:S-HIT_geom} gives: 
$$    \gon(Y_m)\geqslant \frac{N}{\ell}\bigl((1-\eps)\ell\bigr)^{\underline{\lfloor\ell/2\rfloor-1}}\geqslant \frac{N}{\ell}\bigl(\ell/2\bigr)^{\underline{\lfloor \ell/2\rfloor-1}} \geqslant \frac{N}{\ell} \lfloor\ell/2\rfloor !\ \ .$$
   It is easy to verify, by Stirling's formula, that for large $\ell$ the right hand side (RHS) is at least $(N/2)\cdot \exp(4\sqrt{\ell}\log^2\ell)$. Using the explicit bound $n!>(n/e)^n (2\pi n)^{
  1/2}$ and a numerical verification for low degrees shows that one can take $\ell \geqslant 7900$.

Therefore, in all cases, when $\ell\geqslant 7900$, we have $\gon(X_H)>\min\{(1-\eps)N/2,b/8\}$, as needed.  \qedhere
\end{proof}
\begin{cor}\label{cor:preserving-Galois-groups}
    Let $\phi:X\to \mP^1_k$ be a genus $g$ covering of degree $\ell\geqslant 3g$, monodromy $A_\ell$ or $S_\ell$, and $b\geqslant 26$ branch points. 
    Then $\Im\varphi_P\cong A_\ell$ or $S_\ell$ for all but finitely many  points $P$ of degree $d\leqslant \min\{b/16,(\ell-4)/3\}$  outside $\phi(|X|_d)$. If, moreover, $\ell\geqslant 9$, then the Galois group of the fiber is $A_{\ell-1}$ or $S_{\ell-1}$  for all but finitely many points $P \in \phi(|X|_d)$ of degree $d\leqslant \min\{b/16, (\ell-4)/3\}$.
\end{cor}
\begin{proof}
    We apply \Cref{thm:other-actions} with $\eps=1/3$. By \Cref{lem:mystery}, we can take $N=\min\{b-16, 2(\ell-4)\}$. Substituting the values shows that both conclusions hold for 
    \[d<\min\{b/6-8/3, \lfloor b/16 \rfloor+1, (\ell-4)/3\},\]
    which follows from our assumption 
    $d \leqslant \min\{b/16, (\ell-4)/3\}$ (since $b\geqslant 26$).
    Finally, note that we do not need the extra assumption $\ell \geqslant 7$, since for small $\ell$ the conclusion of this corollary is null. Note that in the second claim, the condition $\ell \geqslant 9$ implies the inequality $1/3=\eps \leqslant 1/2-3/(2\ell)$. 
\end{proof}
\section{General Indecomposable maps}\label{section:general-groups}
In this section we focus on primitive coverings of monodromy group $G$; since we have already obtained results for $G=A_\ell, S_\ell$, here we focus on the remaining cases.
The following proposition uses the bounds of Burness--Guralnick. It excludes certain groups of product type which are considered later on in this section. 
\begin{prop}\label{prop:monodromy}
        Let $\tilde \phi:\tilde X\to X_0$ be a Galois covering with monodromy group $G$ and $b$ branch points. If there is an  infinite set $\mathcal S$ of points $P \in |X_0|$ of degree $d<3b/28$ with $\Im\, \tilde\varphi_P\lneq G$, then one of the following holds: 
    \begin{enumerate}
    \item\label{case:quotient} there is a quotient $\Gamma=G/N$ by $1\neq N\lhd G$ for which the induced map $\tilde\psi:\tilde X/N\to X_0$ satisfies: there are  
    infinitely many points $P\in \mathcal S$ such that $\Im\, \tilde\psi_P\lneq \Gamma$;
  
    \item\label{case:altsym} $G=A_\ell$ or $S_\ell$ and, for some $1\leqslant j\leqslant \ell/2$,  the image $\Im\, \tilde \varphi_P$ is contained in a set stabilizer $S_j\times S_{\ell-j}$ for infinitely many $P\in \mathcal S$; 
    \item\label{case:wreath}  $A_\ell^t\leqslant G\leqslant S_\ell\wr S_t$ is primitive of product-type (non-basic) for $\ell\geqslant 5$, $t>1$, and for some $1\leqslant j\leqslant \ell/2$, $\Im\, \tilde \varphi_P$ is contained in the stabilizer $(S_j\times S_{\ell-j})\wr S_t$ of the product-type action with respect to the $S_\ell$-action on $j$-sets. 
   
    \end{enumerate}
\end{prop}
\begin{proof}
 
    By the assumption on $\mathcal S$, there is a (conjugacy class of) a proper subgroup $D<G$ occurring as $\Im\,\varphi_P$ at infinitely many points $P\in |X_0|_d$. Let $M$ be a maximal subgroup $M$ of $G$ containing $D$. The  points $P$  are in the image $\phi_M(|Y|_d)$ of the induced map $\phi_M:Y\to X_0$ from $Y=\tilde X/M$. 
    In particular, $|Y|_d$ is infinite. 
    Since $M$ is maximal, $\phi_M$ is indecomposable. 
    If the Galois closure of $\phi_M$ is a proper subcover of $\tilde \phi$ then (1) holds. 
    
    Thus, assume that the Galois closure of $\phi_M$ is $\tilde \phi$, so that $\Mon(\phi_M)\cong G$ as abstract groups. 
        If  $\Mon(\phi_M)$ is not one of the groups in (2) and (3), 
    then there are only finitely many degree $d$ points in $\phi_M(|Y|_d)$ by Corollary \ref{cor:BGST}, contradicting that $\mathcal S\subseteq |X_0|_d$ is infinite. 
\end{proof}
Note that if option \eqref{case:quotient} occurs, then the proposition can be reapplied iteratively to the quotient (upon changing $b$ appropriately). Option \eqref{case:altsym} is addressed by Theorems \ref{thm:S-HIT}, \ref{thm:other-actions}. 

Option \eqref{case:wreath} is considered in Proposition \ref{prop:product-type} below. We first note:
\begin{lem}\label{lem:eps-prod}
    Let $\phi$ be a map of monodromy group $A_\ell^t\leqslant G\leqslant S_\ell\wr S_t$ of product-type with respect to the degree $m={\ell\choose j}$ action of $S_\ell$ on $j$-sets for $\ell\geqslant 5$, $1\leqslant j\leqslant \ell/2$, and $t\geqslant 2$. 
    Then for $\eps < 2/5$, the branch cycle of every $\eps$-ramified point is contained in $S_\ell^t$. 
\end{lem}
\begin{proof}
    Let $x\in G$ be a  branch cycle of $\phi$ over an $\eps$-ramified point and $\pi:G\to S_t$ be the natural projection. 
    Then $f(x)\leqslant m^{\orb(\pi(x))}$ by \cite[(9.1)]{GT}, where $\orb(\pi(x))$ is the number of orbits of $\pi(x)$ on $\{1,\ldots,t\}$. 
    Since orbits of $x$ that are not counted by $f(x)$ are of length at least $2$, we have $\ind(x)\geqslant (m^t-m^{\orb(\pi(x))})/2$. 
    Suppose $\pi(x)\neq 1$, then  $m\geqslant \ell\geqslant 5$, and so we have     
    \[2/5>\eps \geqslant \frac{\ind(x)}{m^t} \geqslant \frac{m^t-m^{t-1}}{2m^t}=\frac{1}{2}\bigl(1-\frac{1}{m}\bigr)\geqslant \frac{1}{2}\bigl(1-\frac{1}{\ell}\bigr)\geqslant 2/5,\] contradiction. 
\end{proof}

\begin{thm}
    \label{prop:product-type}
Let $\tilde \phi: \tilde X\to X_0$ be
a Galois covering with monodromy group $A_\ell^t\leqslant G\leqslant S_\ell\wr S_t$ for $t\geqslant 2$ and $\ell> 5$. 
Suppose $\phi_j:X_j\to X_0$ is an indecomposable degree $m^t$ subcover of $\tilde \phi$  of product-type action with respect to the $S_\ell$-action on $j$-sets, where $m={\ell \choose j}$  and $1\leqslant j\leqslant \ell/2$. 
Assume that the number of  $\eps$-ramified branch points of $\phi_1$, for some $2/\ell<\eps<2/5$, is at least $N\in\mathbb{N}$, and set $N'\coloneqq (1-\eps)N/2$. We estimate the gonalities of $X_j$ from below, with cases $j=1,2$ requiring special attention. For $X_1$, the following piecewise-linear  bound in  $N'$ holds:
\[\gon(X_1) \geqslant \begin{dcases} \frac{\ell^{t-1}}{(\ell-1)t}N', &\text{ for }N'\leqslant \frac{\ell-1}{(t-1)!};\\
\frac{\ell^{t-1}}{t!}, &\text{ for } \frac{\ell-1}{(t-1)!}\leqslant N' \leqslant \frac{\ell^{t-1}}{(t-1)(t-1)!}; \\
\frac{t-1}{t}N', &\text{ for } \frac{\ell^{t-1}}{(t-1)(t-1)!}\leqslant N' \leqslant \frac{\ell^t}{t-1};\\
\frac{\ell^t}{t}, &\text{ for } \frac{\ell^t}{t-1}\leqslant N'.  
\end{dcases}\]
For the curve $X_2$, we have \[\gon(X_2)  \geqslant    \min\left\{\Bigl(\frac{\ell-1}{2}\Bigr)^t \gon(X_1),\frac{N'}{2}\Bigl(\frac{\ell+1}{2}\Bigr)^{t-1}\right\}.\]
Finally, for all $j\geqslant 3$ we have
\[\gon(X_j) \geqslant
\min\left\{
\frac{1}{j!}\Bigl(\frac{(\ell-2)^{\underline{j-2}}}{2}\Bigr)^t\gon(X_2),\frac{N}{2}\Bigl(\frac{1}{j(\ell-j+1)}{(1-\eps)\ell\choose j-1}\Bigr)^t\right\}.\]
\end{thm}
Recall that $X_{\ell-j}\cong X_j$, so that it indeed suffices to consider $j\leqslant \ell/2$. 
    
The theorem follows from the following four lemmas.
Throughout the proof we use the following notation and setup in addition to the assumptions of \Cref{prop:product-type}. 
\subsection*{Setup}\hypertarget{setup}{}  Set $\Delta=\{1,\ldots,\ell\}$.     Let $H$ denote the image of $G$ in $S_t$ and $K=G\cap S_\ell^t$ the kernel of the projection $G\to H$, so that we have the following commutative diagram.
\begin{center}
\begin{tikzcd}
1 \arrow[r] & K \arrow[r] \arrow[d, hook] & G \arrow[r] \arrow[d, hook] & H \arrow[r] \arrow[d, hook] & 1 \\
1 \arrow[r] &  (S_\ell)^t \arrow[r]       & S_\ell \wr S_t \arrow[r]    & S_t \arrow[r]               & 1
\end{tikzcd}
\end{center}
We view $(S_\ell)^t$, and thus $K$, as acting on $\Delta^t$, while the action of $S_t$ permutes the entries of $\sigma=(\sigma_1, \dots, \sigma_t)\in (S_\ell)^t$. Note that $H$ does not act on $K$, since the upper row does not have to be a semidirect product, but $H$ does naturally act on the $(S_\ell)^t$-conjugacy classes of elements of $K$.

Recall, as in \Cref{Section:permutations-and-RH}, that the primitivity of $G$ implies that $H$ is transitive. let $H_1\coloneqq H\cap S_{t-1}$ be a point stabilizer in the action on $\{1,\ldots,t\}$, and set $h\coloneqq\#H$ and $h_1\coloneqq \#H_1$.  
 Let $Y_j\coloneqq\tilde X/(G\cap (S_{\ell-j}\wr S_t))$, so that the natural projection $Y_j\to X_j$ is of degree at most $(j!)^t$. 
 Since $X_j\cong X_{\ell-j}$ we shall assume $j\leqslant \ell/2$. 
 Let $\pi_j:Y_j\to Y_{j-1}$ be the natural projection of degree $(\ell+j-1)^t$. It is indecomposable by Remark \ref{rem:indec}. Throughout,  we let  $P\in X_0$ be an $\eps$-ramified branch point of $\phi_1$ with branch cycle $x$. Since $\eps<2/5$, we have  $x=(x_1,\ldots,x_t)\in  K\leq S_\ell^t$ by Lemma \ref{lem:eps-prod}. Let $f_i$ be the number of fixed points of $x_i$ on $\{1,\ldots,\ell\}$ for  $i=1,\ldots, t$, so that the number $\prod_{i=1}^tf_i$ of fixed points of $x$ is  at least $(1-\eps)\ell^t$. 

\begin{rem}\label{rem:subsets}  
    We further deduce that $\prod_{i\in S}f_i\geqslant (1-\eps)\ell^{\#S}$ for every subset $S\subseteq \{1,\ldots,t\}$. 
    \end{rem}
    
\begin{lem}\label{lem:k>2}
    For $j\geqslant 3$ one has: 
        $$\gon(X_j)\geqslant \min\Bigl\{\frac{1}{j!}\Bigl(\frac{(\ell-2)^{\underline{j-2}}}{2}\Bigr)^t\gon(X_2),\frac{N}{2}\Bigl(\frac{1}{j(\ell-j+1)}{(1-\eps)\ell\choose j-1}\Bigr)^t\Bigr\}.$$  
 \end{lem}
  \begin{proof}
{We will apply the Castelnuovo--Severi inequality (Proposition \ref{Castelnuovo}) to  $\pi_j$ and a map $Y_j\to \mP^1_k$ of minimal degree. }
We claim inductively that: 
$$\gon(Y_j)\geqslant  \min\Bigl\{\Bigl(\frac{(\ell-2)^{\underline{j-2}}}{2}\Bigr)^t\gon(Y_2), \frac{N}{2}\Bigl(\frac{((1-\eps)\ell)^{\underline{j-1}}}{\ell-j+1}\Bigr)^t\Bigr\}.  $$

     Let $P\in X_0$ be a point with branch cycle $x$ as in the \hyperlink{setup}{Setup}. 
     The fibers of the map $Y_{j-1}\to X_0$ then correspond to the orbits of $x$ on ordered $(j-1)$-tuples. Under this identification, $\pi_j$ sends the orbit of an ordered $j$-tuple to the orbit of its first $j-1$ entries. 
    
    The number of unramified preimages under $Y_{j-1}\to X_0$ is therefore at least 
    $f_1^{\underline{j-1}}\cdots f_t^{\underline{j-1}}$. Since $f_i\geqslant (1-\eps)\ell$ by \Cref{rem:subsets}, this number is at least $f_1^{\underline{j-1}}\cdots f_t^{\underline{j-1}}\geqslant ((1-\eps)\ell)^{\underline{j-1}})^t$. 
    Thus, when considering the orbit of a $j$-tuple whose first $j-1$ entries are fixed by $x$ and the last is not, we get that the Riemann--Hurwitz contribution  over preimages of $P$ in $Y_{j-1}$ is at least $((1-\eps)\ell)^{\underline{j-1}})^t$. Since there are $N$ such points,  $\deg R_{\pi_j}$ is at least 
    $N((1-\eps)\ell)^{\underline{j-1}})^t$.

    Thus, the Castelnuovo--Severi bound from Proposition \ref{Castelnuovo} and the induction hypothesis give:
     \begin{align*}
     \gon(Y_j)  &\geqslant \min \left\{(\ell-j+1)^t\gon(Y_{j-1}), \frac{\deg R_{\pi_j}}{2(\deg\pi_j -1)} \right\}    \\
      &\geqslant \min\Bigl\{\Bigl(\frac{(\ell-2)^{\underline{j-2}}}{2}\Bigr)^t\gon(Y_2),(\ell-j+1)^t\cdot \frac{N}{2}\Bigl(\frac{(1-\eps)\ell)^{\underline{j-2}}}{\ell-j+2}\Bigr)^t,   
     \frac{N}{2}\Bigl(\frac{((1-\eps)\ell)^{\underline{j-1}}}{\ell-j+1}\Bigr)^t\Bigr\}.    \end{align*}
     Since $\eps\geqslant 2/\ell$ and $j\leqslant \ell/2$, the bases of the $t$-exponents satisfy: 
     \begin{equation*}
     (\ell-j+1)\frac{((1-\eps)\ell)^{\underline{j-2}}}{\ell-j+2} > \frac{((1-\eps)\ell)^{\underline{j-1}}}{\ell-j+1}.    
     \end{equation*} 
     and hence the second term in the minimum 
     can be removed, completing the induction.
     
    Since $\gon(X_j)\geqslant\gon(Y_j)/\#H\geqslant \gon(Y_j)/(j!)^t$ and $\gon(Y_2)\geqslant \gon(X_2)$, this gives
    \begin{align*}
        \gon(X_j) & \geqslant \min\left\{\frac{1}{j!}\Bigl(\frac{(\ell-2)^{\underline{j-2}}}{2}\Bigr)^t\gon(Y_2), 
        \frac{N}{2}\Bigl(\frac{{((1-\eps)\ell)}^{\underline{j-1}}}{j!(\ell-j+1)}\Bigr)^t \right\} \\
    & = \min\left\{\frac{1}{j!}\Bigl(\frac{(\ell-2)^{\underline{j-2}}}{2}\Bigr)^t\gon(X_2), 
    \frac{N}{2}\Bigl(\frac{1}{j(\ell-j+1)}{(1-\eps)\ell\choose j-1}\Bigr)^t\right\}. \qedhere
    \end{align*} 
    \end{proof}

    For $j=2$, we carry out a more careful analysis:
    
    \begin{lem}\label{lem:k=2}

\[\gon(X_2) \geqslant         
    \min\left\{\Bigl(\frac{\ell-1}{2}\Bigr)^t\gon(X_1),\frac{N'}{2}\Bigl(\frac{\ell+1}{2}\Bigr)^{t-1}\right\}.\]     
    \end{lem}
    \begin{proof}
    We use the above \hyperlink{setup}{Setup}. Let $\iota$ be an index such that $f_\iota<\ell$. 
    Then the number of fixed points of $x$ is at least $\prod_if_i$, while the number of points moved is at least $(\ell-f_\iota)\ell^{t-1}$.

 Thus, the Riemann--Hurwitz contribution to $Y_2\to Y_1$ over preimages of $P$ is at least $\frac{1}{2}(\ell-f_\iota)\ell^{t-1} \prod_{i=1}^tf_i$. 
    Since $f_\iota$ are integers, we have $f_\iota(\ell-f_\iota)\geqslant \ell-1$; in addition $\prod_{i \neq \iota} f_i\geqslant (1-\eps)\ell^{t-1}$, and since the number of $\eps$-ramified points is at least $N$,  we get the lower bound:
    $$ \deg R_{\pi_k}\geqslant N \ell^{t-1}\cdot \frac{(\ell-f_\iota)f_\iota}{2} \prod_{i\neq \iota}f_i\geqslant \frac{N}{2}\ell^{t-1}(\ell-1)(1-\eps)\ell^{t-1}=N'(\ell-1)\ell^{2t-2}.$$
       
    Thus, applying Proposition \ref{Castelnuovo} (Castelnuovo--Severi) to $\pi_2$, we have 
    \begin{align*}
        \gon(X_2)\geqslant \frac{\gon(Y_2)}{2^t} & \geqslant  \min\left\{\frac{(\ell-1)^{t}}{2^t}\gon(Y_1), \frac{N'(\ell-1)\ell^{2t-2}}{2^{t}(\ell-1)^t}\right\} \\ 
        & \geqslant \min\left\{\Bigl(\frac{\ell-1}{2}\Bigr)^t\gon(X_1),\frac{N'}{2}\Bigl(\frac{\ell+1}{2}\Bigr)^{t-1}\right\}.
    \end{align*}  
    
    \end{proof}
We now bound $\gon(X_1)$:
\begin{lem}\label{lem:gon(X_1)}
   
    \begin{align*}
    \gon(X_1) & \geqslant  \min\left\{\frac{\ell^{t-1}}{t!}\gon(X_0),\frac{N'\ell^{t-1}}{(\ell-1)t}\right\}.
    \end{align*}
\end{lem}
\begin{proof}
We use the above \hyperlink{setup}{Setup}. Let $Z_0=\tilde X/K$ and $Z_t=\tilde X/(K\cap S_{\ell-1}^t)$, so the projection $Z_t\to Z_0$ is of degree $\ell^{t}$. The map $Z_t\to Z_0$ factors through the sequence of degree $\ell$ maps $\varpi_i:Z_i\to Z_{i-1}$, $i=1\ldots,t$, where $Z_i=\tilde X/K\cap (S_{\ell-1}^i\times S_\ell^{t-i})$. 

We next bound the Riemann--Hutwitz contribution  $\deg R_{\varpi_i}$. As in the setup, let $P$ be an $\eps$-ramified branch point of $\phi_1$ with branch cycle $x=(x_1,\ldots,x_t)\in K$. Hence, it has $h$ preimages $Q$ in $Z_0$. 
The regular action of $H$ is equivalent to its action on these preimages. Moreover, if $x$ is the branch cycle over $Q_0$ and $Q=h(Q_0)$ for $h\in H$, then\footnote{This standard observation is also made in \cite[Lemma 4.1]{GN} and \cite[\S 2.7]{NZ2}.} the branch cycle over $Q$ is conjugate to $gxg^{-1}$  for any lift $g\in G$ of $h$. 
Since this conjugation action permutes the $t$ coordinates in $S_\ell^t$,
there are at least $h_1$ such points $Q$ whose branch cycle $y\in K$ has a nontrivial $i$-th coordinate. 
Note that the points in $Z_i$ over such $Q$ are in correspondence with orbits of $y$ on $\Delta^i$.
Since $y$ is conjugate to $x$, it has at least $(1-\eps)\ell^{i-1}$ fixed points on $\Delta^{i-1}$ by Remark \ref{rem:subsets}. Orbits of $y$ of length $\geq 2$ on $\Delta^{i}$ that project to a fixed point in the action on $\Delta^{i-1}$ then correspond to ramification points of $Z_i\to Z_0$ whose images in $Z_{i-1}$ are unramified over $Z_0$. 
Thus, there are at least $(1-\eps)\ell^{i-1}$ points of $Z_{i-1}$ over $Q$ which are branch points of $\varpi_i$. Thus, in total 
  $\varpi_i$ has at least $N\cdot h_1\cdot (1-\eps)\ell^{i-1}$ branch points, and hence  $\deg R_{\varpi_i}$ is at least that amount. 

We next inductively claim that $\gon(Z_i)\geqslant \min\{\ell^{i-1}\gon(X_0),N'h_1\ell^{i-1}/(\ell-1)\}$ for $i=2,\ldots,t$. 
For $i=2$, the above estimates give: 
\begin{align*}
    \gon(Z_2)\geqslant \min\left\{\ell\gon(Z_1),\frac{\deg R_{\varpi_2}}{2(\ell-1)}\right\}\geqslant \min\left\{\ell\gon(X_0),\frac{N'h_1\ell}{\ell-1}\right\}. 
\end{align*}
Assuming the claim for $Z_{i-1}$, we similarly get by induction: 
\begin{align*}
    \gon(Z_i)\geqslant \min\left\{\ell\gon(Z_{i-1}),\frac{\deg R_{\varpi_i}}{2(\ell-1)}\right\}\geqslant \min\left\{\ell^{i-1}\gon(X_0),\frac{N'h_1\ell^{i-1}}{\ell-1}\right\}, 
\end{align*}
as claimed. As $\gon(X_1)\geqslant \gon(Z_t)/h$ by \Cref{lem:gon-and-airr} and $h/h_1=t$ (as $H$ is transitive of degree $t$), we have:
\begin{align*}
    \gon(X_1)\geqslant 
    \min\left\{\frac{\ell^{t-1}}{h}\gon(X_0),\frac{N'h_1\ell^{t-1}}{(\ell-1)h} \right\} 
    = \min\left\{\frac{\ell^{t-1}}{h}\gon(X_0),\frac{N'\ell^{t-1}}{(\ell-1)t}\right\}.
\end{align*}
The lower bound on $\gon(X_1)$ now follows as $h\leqslant t!$. 
\end{proof}

We next bound $\gon(X_1)$ also in cases where $t$ is allowed to be large in comparison to $\ell$: 
\begin{lem}\label{lem:gon(X_1)b} 
    $$\gon(X_1)\geqslant \min\left\{\frac{\ell^t}{t}\gon(X_0), \frac{N'(t-1)}{t}\right\}.$$
    
\end{lem}
\begin{proof}
We apply an argument similar to that in Lemma \ref{lem:gon(X_1)} for a factorization, through a tower of covers, of the map from  $Y_t' \coloneqq \tilde X/G\cap (S_{\ell-1}^t\cdot S_{t-1})$ to $Y_0'\coloneqq \tilde X/G\cap (S_\ell^t\cdot S_{t-1})$.
Let $O_1, \ldots, O_m$ be the orbits of $H_1$  
on $T\coloneqq \{1,\ldots,t\}$, ordered so that 
  $o_1\leqslant \cdots\leqslant o_{m}$ where $o_i\coloneqq \#O_i$.  In particular, $o_1=1$. 
Let $Z_i$ be the quotient of $\tilde X$ by the point stabilizer in the action of $G\cap (S_\ell^t\cdot S_{t-1})$ on $\Delta^{\bigcup_{j=1}^iO_j}$ for $i=0,\ldots,m$. 
In particular, we have $Z_0=Y_0'$ 
and $Z_m=Y_t'$. 
Let $\varpi_i:Z_i\to Z_{i-1}$ be the natural projection and note that $\varpi_i$ is indecomposable for $i=1,\ldots,m$ by Remark \ref{rem:indec}.

We claim inductively on $1\leqslant i\leqslant m$ that 
$$\gon(Z_i)\geqslant \min\left\{\ell^{\sum_{j=1}^io_j}\gon(Z_0), N'o_i\ell^{-1+\sum_{j=1}^{i-1}o_j}\right\}.$$
For $i=1$, since the Riemann--Hurwitz contribution over each of the $N$ $\eps$-ramified points is at least  $(1-\eps)\ell$, we have  $\gon(Z_1)\geqslant \min\{\ell\gon(Z_0),N(1-\eps)\}$ by   \Cref{Castelnuovo}. 

Henceforth assume the lower bound on $\gon(Z_{i-1})$ for $i\geqslant 2$. To bound the Riemann--Hurwitz contribution $\deg R_{\varpi_i}$, 
note that there are at least $N$ $\eps$-ramified points $P$ with branch cycles $x\in K$. 
Each such point $P$ has $t$ preimages $Q$ in $Z_0$ and the action of $H$ on these is equivalent to its action on $\{1,\ldots,t\}$. 
If $Q_0$ is a point of $Z_0$ with branch cycle $x$ and $Q=\sigma(Q_0)$ for $\sigma\in G$, then the branch cycle over $Q$ is conjugate to $\sigma x\sigma^{-1}$ in the preimage $K.H_1\leq G$ of $H_1$. As $x\neq 1$, we deduce there are at least $o_i$ such places $Q$ with branch cycle $y$ admitting at least one nontrivial coordinate in $O_i$. 
Each such $y$ moves one of the $O_i$ coordinates and has at least $(1-\eps)\ell^{(-1+\sum o_j)}$ fixed points on the rest of the coordinates in  $\cup_j O_j$ by Remark \ref{rem:subsets}, where $j$ ranges over $1,\ldots,i$. 
Thus  $\deg R_{\varpi_i}$ is at least $No_i(1-\eps)\ell^{(-1+\sum_{j=1}^i o_j)}=2N'o_i\ell^{(-1+\sum_{j=1}^i o_j)}$.  

Proposition \ref{Castelnuovo} and the induction hypothesis therefore give: 
\begin{align*}
    \gon(Z_i) & \geqslant \min\left\{\ell^{o_i}\ell^{\sum_{j=1}^{i-1}o_j}\gon(Z_0), \ell^{o_i}N'o_{i-1}\ell^{-1+\sum_{j=1}^{i-2}o_j}, 
    \frac{2N'o_{i}\ell^{-1+\sum_{j=1}^{i}{o_j}}}{2(\ell^{o_i}-1)} \right\} \\ 
    & \geqslant \min\left\{\ell^{\sum_{j=1}^{i}o_j}\gon(Z_0), N'o_{i-1}\ell^{-1-o_{i-1}+\sum_{j=1}^{i}{o_j}}, N'o_i\ell^{-1-o_i+\sum_{j=1}^{i}{o_j}}\right\}. 
\end{align*}
For $\ell\geqslant 5$, since the function $x\ell^{-x}$ is decreasing for all $x\geqslant 1$, and $o_i\geqslant o_{i-1}$, the third term is at most the second so that the induction claim holds.

Since in addition  $\gon(X_1)\geqslant\gon(Y_t')/t$ by \Cref{lem:gon-and-airr}, and $o_m\leqslant t-1$,  and since $x\ell^{-x}$ is decreasing, we have:
$$ \gon(X_1)\geqslant \min\left\{\frac{\ell^t}{t}\gon(Z_0),\frac{N'o_m\ell^{t-o_m-1}}{t}\right\}\geqslant \min\left\{\frac{\ell^t}{t}\gon(Z_0),N'\frac{t-1}{t}\right\}.$$
The claim now follows since $\gon(Z_0)\geqslant \gon(X_0)$ by  \Cref{lem:gon-and-airr}. 
\end{proof}
\begin{proof}[Proof of \Cref{prop:product-type}]
    The bounds on $\gon(X_k)$ for $k\geqslant 2$ are given in Lemmas \ref{lem:k>2} and \ref{lem:k=2}. The bounds on $\gon(X_1)$ are obtained by comparing the bounds from Lemmas \ref{lem:gon(X_1)} and \ref{lem:gon(X_1)b} on the corresponding intervals. 
\end{proof}

Finally, we apply the proposition to study preservation of Galois groups of product-type $A_\ell^t\leqslant G\leqslant S_\ell\wr S_t$ under specialization. To simplify notation, we no longer follow the above \hyperlink{setup}{Setup}. 
Similarly to Theorem \ref{thm:Galois-groups} for  
$S_\ell$ where the minimal normal subgroup $A_\ell$ was avoided (as a case of its own), here  we avoid considering subgroups that contain the minimal normal subgroup $A_\ell^t$. Moreover, we avoid subgroups $H$ whose image  under  $\pi:S_\ell\wr S_t\to S_2\wr S_t$ is smaller than the image $\pi(G)$ of $G$. For such groups the theorem can be applied to a smaller group  of product-type (namely, either to $\pi^{-1}(\pi(H))$ if $H$ is transitive on $\{1,\ldots,t\}$, or to its actions on orbits on $\{1,\ldots,t\}$). 
\begin{cor}\label{cor:product-type-arithmetic}
Suppose $\phi: X_1 \to X_0$ is a map of monodromy   $A_\ell^t\leqslant G\leqslant S_\ell\wr S_t$ of product-type in $t$-tuples from $\{1,\ldots,\ell\}$ for $\ell> 5$ and $t\geqslant 2$.  
Suppose $\phi$ has at least $b$ branch points and  at least $N$ $\eps$-ramified points for some $2/\ell<\eps<2/5$. 
Let $\pi:S_\ell\wr S_t\to S_2\wr S_t$ be the natural quotient modulo $A_\ell^t$. Then for all but finitely many  points $P$  of degree $d<\min\{(1-\eps)N/8,3b/56,\ell^t/(2t)\}$ such that $\pi(\Im\varphi_P)=\pi(G)$, one has  $\Im\varphi_P=G$.   
\end{cor}
\begin{proof}
    Assume  to the contrary that there are infinitely many degree $d$ points $P$  of $X_0$ such that $\Im\varphi_P$ is some proper subgroup $H\lneq G$ satisfying $\pi(H)=\pi(G)$.
    Thus,  $H$ embeds in a maximal subgroup $M\lneq G$ such that $\pi(M)=\pi(G)$. In particular,  $M$ does not contain $\ker\pi = A_\ell^t$.  Set $n\coloneqq[G:M]\geqslant \ell$ and $N'=N(1-\eps)$. 

    As $d<3b/56$,   $M$ has to be a stabilizer in the degree ${\ell \choose j}^t$ product-type action of $G$ on $t$-tuples of $j$-sets by Proposition  \ref{prop:monodromy} applied to the cover $\tilde X/M\to\mP^1_k$. 
    The gonality of $\tilde X/M$ and hence that of $\tilde X/H$ is then at least the lower bound on $\gon(X_j)$ given in  \Cref{prop:product-type}. By \Cref{lem:gon(X_1)b}, we have $\gon(X_1)\geqslant \min\{(t-1)N'/t, \ell^t/t\}$; a direct comparison with other lower bounds from \Cref{prop:product-type} shows that the same lower bound holds for $\gon(X_j)$ for $j\geqslant 2$. As $t\geqslant 2$, we get $\gon(\tilde X/H)\geqslant \min\{N'/2, \ell^t/t\}$ and hence $d\geqslant \min\{N'/4,\ell^t/(2t)\}$ by Proposition \ref{prop:arithm-gonality}, contradicting our assumption on $d$. \qedhere    
\end{proof}

\appendix
\section{Fibers above symmetric points}\label{appendix}

Our goal is to supplement the results of the paper by showing that the Galois structure of fibers above high-degree points is arbitrary. This can be proved in greater generality for covers of varieties. We start by discussing this higher-dimensional setup.

Suppose $X$ is a connected scheme, $\bar{y}$ is a geometric point of $X$, and $\pi_G:\tilde{X} \to X$ is a finite \'etale Galois $G$-covering. Such a covering corresponds to a morphism  $\pi_1^\et(X, \bar{y}) \twoheadrightarrow G$. For every closed point $x \in |X|$ and a geometric point $\bar{x}$ above $x$, choose a path from $\bar{x}$ to $\bar{y}$ and consider the composition $\phi_x: \pi_1^{\et}(\Spec k(x), \bar{x}) \to \pi_1^{\et}(X, \bar{x}) \xrightarrow{\sim} \pi^{\et}_1(X, \bar{y}) \twoheadrightarrow G$. The image $H_x = \mathrm{im}\ \phi_x$ of this map describes the Galois action on the fiber above $x$; changing the path from $\bar{x}$ to $\bar{y}$ changes $H_x$ by a conjugation, and so every closed point $x \in |X|$ defines a conjugacy class of a subgroup $H_x \subset G$. We consider the case when $X$ is an arbitrary smooth variety over a number field $k$ (a variety is a separated scheme of finite type over a field). In this setting the set of closed points $|X|$ is too complicated to talk about distribution of Galois groups in the analytic sense; our main result is that every conjugacy class appears above some point $x \in |X|$, and that, moreover, we can require $k(x)/k$ to be a degree $n$, $S_n$-extension for all  $n$ divisible by a sufficiently large integer.  

Recall that the \emph{index} $i(X)$ of a variety $X/k$ is the greatest common divisor of the degrees $\deg P$ of closed points $P \in |X|$. A degree $n$ separable field extension $K/k$ is called an $S_n$-extension if the Galois group of the Galois closure of $K/k$ is the symmetric group $S_n$; we call degree $n$ closed point $P$ on a variety $X/k$ an \emph{$S_n$-point} if $k(x)/k$ is an $S_n$-extension.

\begin{thm}\label{main theorem}
	Suppose $k$ is a finitely generated field of characteristic $0$, and $X/k$ is a smooth quasi-projective variety of dimension at least $1$. Suppose $\tilde{X} \to X$ is a finite \'etale Galois covering with Galois group $G$ such that $\tilde{X}$ is geometrically irreducible. Fix a subgroup $H \subset G$. Then there exists a constant $N$ such that for any finite extension $L/k$ and for any $n>N$ which is divisible by $i(X)[G:H]$, there exist infinitely many degree $n$ $S_n$-points $x \in |X_L|$ such that $H_x$ is conjugate to $H$. 
\end{thm}
\begin{rem}
	In this theorem we may replace $X$ with an open subset that still contains a degree $i(X)$ point. In particular, one can replace the quasi-projective assumption with the notion of ``FA-scheme'' in the sense of \cite[Section 2.2]{gabber2013index}.
\end{rem}
\begin{rem}
	Theorem \ref{main theorem} is related to a result of Poonen \cite[Theorem 1]{poonen2001points}, which implies that, in the notation of our theorem, there is a closed point $x \in |X|$ with $H_x = \{e\}$ (but does not give control over the field extension $k(x)/k$.) 
\end{rem}
\begin{rem}
The case $H=G$ of the theorem is also known to hold, for any Hilbertian field $k$. In the language of field arithmetic (field stability) this was proved by many authors; see \cite[Theorem 18.9.3]{fried2005field}. This appendix extends the result to an arbitrary subgroup $H \subset G$.
\end{rem}

This result can be applied to a covering of moduli spaces to build objects with prescribed Galois actions. As an example, we show how to construct abelian varieties with specified level $N$ structures.
\begin{thm}\label{level-structures}
	Suppose $k$ is a number field that contains $N$-th roots of unity, $g\geqslant 1$ is an integer, and $H \subset \mathrm{Sp}(2g, \Z/N\Z)$ is a subgroup of index $d$. Then there exists a constant $M=M(H)$ such that if $n>M$ is an integer divisible by $d$, then there exists a degree $n$, $S_n$-extension $K/k$ and a $g$-dimensional abelian variety $A/K$ such that the image of the Galois action on the $N$-torsion points $\Gal_K \to \mathrm{Sp}(A[N])$ is conjugate to $H$.
\end{thm}

We first reduce the proof of Theorem \ref{main theorem} to the case of curves by combining Bertini's and Lefschetz's theorems.

\begin{lem}\label{reduce-to-curves}
	In the setting of Theorem \ref{main theorem}, let $D \subset X$ be a closed subscheme which is a union of closed points $D=\bigcup_{i=1}^s P_i$ such that $\gcd(\deg P_i)=i(X)$. Then there exists a smooth geometrically integral curve $Z \subset X$ with $D \subset Z$ such that for a point $z \in (Z \setminus D)(\mC)$ the natural map $\pi_1((Z\setminus D)(\mC), z) \to \pi_1((X\setminus D)(\mC), z)$ is surjective. 
\end{lem}
\begin{proof}
	We prove the statement by induction on $\dim X$. The base case $\dim X = 1$ is immediate. If $\dim X>1$, we embed $X$ into a projective space $\mP^r$ and consider the intersection of $X$ with a general degree $n$ hypersurface $H \subset \mP^r$ passing through $D$. For $n$ large enough, this intersection $X'=X \cap H$ is smooth and geometrically irreducible by a suitable version of Bertini's theorem (see, for example, \cite[Theorems 1 and 7]{kleiman1979bertini}). The surjectivity on fundamental groups (for sufficiently large $n$) follows from a version of Lefschetz's theorem, as we now explain. After replacing $X \to \mP^r$ with a high-degree Veronese embedding, we can assume that the span $\Span D$ of the points of $D$ intersects $X$ only at $D$, and that the projection $\pi_D$ from $\Span D$ has $\dim X$-dimensional image. In this language, the variety $X'\setminus D$ is a preimage of a hyperplane under $\pi_D$, and for $x \in X'\setminus D$ the surjectivity of $\pi_1((X'\setminus D)(\mC),x) \to \pi_1((X\setminus D)(\mC),x)$ follows from Lefschetz theorem in the form of \cite[Lemma 1.4]{deligne1981groupe}.
\end{proof} 
\begin{rem}
	Since the varieties are assumed to be smooth, if $\dim X > 1$ we have a natural isomorphism $\pi_1((X\setminus D)(\mC),z) \xrightarrow{\sim} \pi_1(X(\mC),z).$
\end{rem}

\begin{proof}[Proof of Theorem \ref{main theorem}]
Fix a collection of distinct closed points $P_1, \dots, P_s \in |X|$ such that $\gcd(\deg P_i)=i(X)$, and denote by $D$ the union $D=\bigcup_i P_i$. Let $Z$ be the smooth curve from Lemma \ref{reduce-to-curves}. Since the map $\pi_1((Z\setminus D)(\mC), z) \to \pi_1((X\setminus D)(\mC), z)$ is surjective, the covering $\tilde{X} \to X$ remains a geometrically connected Galois $G$-covering when pulled back to $Z$. Therefore it suffices to consider the case $\dim X =1$. In this case, after compactifying, we can assume that $\tilde{X}$ and $X$ are both smooth proper curves, $\pi_G: \tilde{X} \to X$ is a geometrically irreducible $G$-covering, and $D \subset X$ is a divisor which does not intersect the branch locus of $\pi_G$. 

Let $\pi_H:X_H \to X$ be the intermediate covering of $\pi_G$ corresponding to $H$, so that $\pi_{G/H}:\tilde{X} \to X_H$ is an $H$-covering. Another way of phrasing the theorem is that there are infinitely many degree $n$, $S_n$-points $x \in |X_L|$ and $L(x)$-rational points $x_H \in \pi_H^{-1}(x)$ such that $\pi_{G/H}^{-1}(x_H)$ is an irreducible scheme. Consider the collection $\calS$ of divisors of the form $E=m_1 \pi_H^{-1}(P_1) + \dots + m_s \pi_H^{-1}(P_s)$ for nonnegative integers $m_1, \dots, m_s$. Let $m$ be a constant satisfying the following conditions:
\begin{enumerate}
	\item $m > [G:H]=\deg \pi_H$;
	\item  $m>2 g(X_H)$, so that any divisor on $X_H$ of degree larger than $m$ is very ample;
	\item any integer larger than $m$ and divisible by $i(X)[G:H]$ is the degree of a divisor from $\calS$ ($m$ is larger than the Frobenius number of the semigroup of degrees of divisors from $\calS$).
\end{enumerate}
Note that $m$ can be chosen independently of $L.$ We claim that any $n>m$ and divisible by $i(X)[G:H]$ satisfies the conditions of the theorem. Fix an $n>m$ divisible by $i(X)[G:H]$ and choose a divisor $E \in \calS$ of degree $n$. Consider the embedding $X_H \subset \mP^r$ given by the complete linear system $|E|$. 
 
 To simplify notation, for the remainder of the proof all varieties are considered after a base change to $L$.
 Consider the correspondences $I_G$, $I_H$ that parameterize incidences between points $x$ on $\tilde{X}$ and $X_H$ and hyperplanes $h$ in $\mP^{|E|}$: $I_G \subset \tilde{X} \times \left(\mP^{|E|}\right)^\vee$ and $I_H \subset X_H \times \left(\mP^{|E|}\right)^\vee$, given by
 \[I_G=\{(x, h)\colon \pi_{G/H}(x) \in h\}\] and \[I_H=\{(x, h)\colon x \in h\}.\]
 The variety $I_G$ is irreducible, since the projection $I_G \to \tilde{X}$ is as a proper map with irreducible equidimensional fibers (the fibers are projective spaces of dimension $\dim |E|-1$). The covering $I_H \to (\mP^{|E|})^\vee$ has degree $n$ and monodromy group $S_n$ (see, for example,  \cite[Lemma, Chapter III, page 111]{ACGH}). Therefore, by Hilbert's irreducibility theorem applied to the (factored) covering $I_G \to I_H \to \left(\mP^{|E|}\right)^\vee$, a general hyperplane section $x_H = h \cap X_H$ is a degree $n$, $S_n$-point on $X_H$, and moreover, since $I_G$ is irreducible, the fiber $\pi_{G/H}^{-1}(x_H)$ is irreducible as well (see \cite[Chapter 9]{serre1989lectures} for Hilbert's theorem in a geometric form). Finally, consider the image $x \coloneq \pi_H(x_H)$. The field extension $L(x_H)/L$ has no intermediate subextensions, and so either $L(x)=L$, or $L(x)=L(x_H)$. If $x$ is a rational point, then the degree of $x_H \subset \pi_H^{-1}(x)$ is at most $[G:H]$, contradicting the assumption $n\geqslant m > [G:H]$. Therefore $x$ is the degree $n$, $S_n$-point on $X$ we seek. 
\end{proof}

\begin{proof}[Proof of Theorem \ref{level-structures}]
Consider any abelian scheme $\calA \to X$ over a smooth base $X$ equipped with a geometric point $\bar{x}$ above a rational point $x \in X(k)$ such that the action of the fundamental group $\bar{\rho}_N\colon \pi_1^{\et}(X_{\kbar}, \bar{x}) \to \mathrm{Sp}(2g, \Z/N\Z)$ on the $N$-torsion is surjective. There are many different such families for which the monodromy is known to be surjective, for instance it is true for the Jacobian of the universal curve (see \cite[Section 5.12]{deligne1969irreducibility}). Since $k$ is assumed to have $N$-th roots of unity, the arithmetic monodromy --- image of $\rho_N:\pi_1^{\et}(X, x) \to \GL(A[N])$ --- is contained in $\mathrm{Sp}(A[N])$. Since $\bar{\rho}_N$ is surjective, so is the arithmetic monodromy $\rho_N$. Therefore the covering $X_N \to X$ corresponding to the kernel of $\rho_N$ is a geometrically irreducible Galois \'etale $\mathrm{Sp}(2g, \Z/N\Z)$-covering of algebraic varieties. Applying Theorem \ref{main theorem} to this covering gives the result.  
\end{proof}
\section{An extension of a theorem of Song and Tucker}\label{appendix:lemma}

In this appendix we explain how to extend \cite[Proposition 2.3]{song2001arithmetic} from a number field to a finitely generated field of characteristic zero.
\begin{prop}[Proposition 2.3 of \cite{song2001arithmetic}]\label{song-tucker-vintage}
Suppose $k$ is a number field, and $\phi:X \to X_0$ is a degree $\ell$ covering of curves. If 
\[\ell \chi(X_0) - \chi(X)>d\ell,\]
then the set \[\{P \in |X|:\ \deg P = \deg \phi(P)=d\}\] is finite.
\end{prop}
\begin{lem}\label{rem:number-fields}
     \Cref{song-tucker-vintage} holds over any finitely generated field $K$ of characteristic zero (in place of $k$).
\end{lem}
\begin{proof}
Suppose $\phi:X \to X_0$ is a degree $\ell$ covering of curves over $K$, the inequality $\ell \chi(X_0)-\chi(X)>d\ell$ holds, and yet the set $\calS_d=\{P \in |X|\colon \deg P = \deg \phi(P)=d\}$ is infinite. We will derive a contradiction with \Cref{song-tucker-vintage} by spreading out and specializing to a number field. See \cite[Section 3.2]{poonen2023rational} for an introduction to spreading out. Note that, as in \Cref{rem:song-tucker-geom-irr}, the curves $X, X_0$ can be assumed to be geometrically integral. We first reinterpret geometrically the infinitude of $\calS_d$. Consider the symmetric powers $\Sym^d X$ and $\Sym^d X_0$ and the induced map $\phi^{(d)}:\Sym^d X \to \Sym^d X_0$; both $\Sym^d X$ and $\Sym^d X_0$ are smooth proper geometrically integral varieties over $K$ (in particular, $\phi^{(d)}$ is a proper map). Let $\Delta_0\subset \Sym^d X_0$ denote the big diagonal, so that the complement $\Sym^d X_0 \setminus \Delta_0$ can be thought of as parameterizing unordered $d$-tuples of distinct points in $X_0$. Let $\Delta=(\phi^{(d)})^{-1}(\Delta_0)$.  The infinitude of $\calS_d$ implies that the set $\calS=(\Sym^d X\setminus\Delta) (K)$ is infinite. Consider the Abel--Jacobi map $\pi_{\mathrm{AJ}}\colon \Sym^d X \to \Pic^d_X $. The fibers of $\pi_{\mathrm{AJ}}$ above rational points are either pointless, or isomorphic to a projective space. Thus either a point of $\calS$, when viewed as an effective divisor, moves in a pencil, or the map $\pi_{\mathrm{AJ}}\colon \calS \to \Pic^d_X(K)$ is injective. We will consider these two cases separately.

{\bf Case 1:} There exists an effective divisor $D\in (\Sym^d X \setminus \Delta)(K)$ which moves in a pencil. Then there exists a map $\rho : X \to \mP^1_k$ of degree $e\leqslant d$ and such that the fiber $\rho^{-1}(\infty)$ is a subset of $D$. Since $D \notin \Delta$, this means that the maps $\rho$ and $\phi$ do not factor through a shared subcover, or equivalently, the morphism $(\phi, \rho): X \to X_0 \times \mP^1_{k}$ is birational onto its image. This contradicts the Castelnuovo--Severi inequality (\Cref{Castelnuovo}.)

{\bf Case 2:} The map  $\pi_{\mathrm{AJ}}\colon \calS \to \Pic^d_X (K)$ is injective. In this case, by Faltings' theorem, there is a coset $A \subset \Pic^d_X$ of an abelian variety of positive rank such that $A$ belongs to the Zariski closure $\pi_{\mathrm{AJ}}(\calS)$. Since fibers above points of $\pi_{\mathrm{AJ}}(\calS)$ are single points, there is an open subset $U\subset A$ and an injective morphism $U \to \Sym^d X$. Fix a pair of rational points $P, Q \in U(K)$ such that $P-Q$ is nontorsion and view $A$ as an abelian variety with origin at $Q$. 

Now spread out $X, X_0, A, U, \phi, P$; this gives an irreducible scheme $S$, whose function field is $K$ (so $S$ is a variety over a number field), and relative curves (smooth proper morphisms of relative dimension $1$) $\calX, \calX_0/ S$, an abelian scheme $\calA/S$ and an open subscheme $\calU\subset \calA$ which surjects onto $S$, a finite morphism $\varphi:\calX \to \calX_0$, over $S$, a section $\calP: S \to \calU$, and a morphism $\psi$ from $\calU$ to the relative symmetric power $\Sym^d_S X$ (so that the restrictions of $\calX, \calX_0, \calA, \calU, \varphi, \calP$ to the generic fiber are $X, X_0, A, U, \phi, P$). Let $\varDelta_0$ denote the relative big diagonal of $\calX_0$ and let $\varDelta$ denote its preimage in $\Sym^d_S \calX$. We can demand two additional properties, by shrinking $S$ further, if necessary. First, we can assume that for every point $s \in S$ the induced map on the fiber $\psi_s: \calU_s\to \Sym^d \calX_s$ is not constant. Secondly, we can assume that $\psi_s(\calP_s)$ does not belong to $\Delta$. 

By Masser's theorem on specialization of subgroups \cite[Main Theorem]{masser1989specializations}, since $P-Q$ is nontorsion, we can find a closed point $s\in S$ such that $\calP_s \in \calA_s$ is not a torsion point (in fact this holds for ``most'' $s \in S$ in a suitable sense). Now we specialize everything to this point $s$. Consider the curve $\calX_s$ over the number field $k=k(s)$. Let $B \subset \calA_s$ denote an abelian coset which contains $\calP_s$ and has dense rational points. Consider the set $\calS'=\{p \in B(k)\colon \psi_s(p) \notin \Delta_s\}$; it is infinite since rational points are dense in $B$ and the set $\calS'$ contains $\calP_s$. Points in $\psi_s(\calU_s\cap\calS')$ correspond to degree $d$ divisors $D$ on $\calX_s$ which consist of distinct points, and such that $\varphi_s(D)$ has degree $d$ and consists of distinct points. Consider now the degree $\ell$ covering of curves $\varphi_s:\calX_s \to (\calX_0)_s$ over the number field $k(s)$. The Euler characteristics are preserved under specialization, and so $\ell \chi((\calX_0)_s)-\chi(\calX_s)>d\ell$; at the same time we have just shown that the set $\calS'_{\leqslant d}=\{P \in |\calX_s|\colon \deg P = \deg \varphi_s(P)\leqslant d\}$ is infinite. This contradicts \Cref{song-tucker-vintage}.
\end{proof}

\bibliographystyle{amsalpha}
\bibliography{Bibliography.bib}

@article{abramovich1991abelian,
  title={Abelian varieties and curves in $ {W}_{d}({C})$},
  author={Abramovich, Dan and Harris, Joe},
  journal={Compositio Mathematica},
  volume={78},
  number={2},
  pages={227--238},
  year={1991}
}

@book{poonen2023rational,
  title={Rational points on varieties},
  author={Poonen, Bjorn},
  series={Graduate Studies in Mathematics},
  volume={186},
  year={2023},
  publisher={American Mathematical Society}
}

@article{poonen2001points,
  title={Points having the same residue field as their image under a morphism},
  author={Poonen, Bjorn},
  journal={Journal of Algebra},
  volume={243},
  number={1},
  pages={224--227},
  year={2001},
  publisher={Elsevier Science}
}

@article{kleiman1979bertini,
  title={Bertini theorems for hypersurface sections containing a subscheme},
  author={Kleiman, Steven L and Altman, Allen B},
  journal={Communications in Algebra},
  volume={7},
  number={8},
  pages={775--790},
  year={1979},
  publisher={Taylor \& Francis}
}

@article{bilu2026values,
  title={Values of algebraic functions at {L}iouville numbers},
  author={Bilu, Yuri and Marques, Diego},
  journal={arXiv:2604.08818},
  year={2026}
}

@misc{lmfdb,
  shorthand    = {LMFDB},
  author       = {The {LMFDB Collaboration}},
  title        = {The {L}-functions and modular forms database},
  howpublished = {\url{https://www.lmfdb.org}},
  year         = {2026},
  note         = {[Online; accessed 19 February 2026]},
}

@article{fehm2025hilbert,
  title={Hilbert properties of varieties},
  author={Fehm, Arno and Javanpeykar, Ariyan},
  journal={arXiv:2511.18431},
  year={2025}
}

@article{kani1984castelnuovo,
  title={On {C}astelnuovo's equivalence defect},
  author={Kani, Ernst},
  journal={Journal f{\"u}r die reine und angewandte Mathematik},
  volume={346},
  pages={24--70},
  year={1984}
}

@article{NZ2,
  title={Monodromy groups of product type},
  author={Neftin, Danny and Zieve, Michael E},
  journal={arXiv:2403.17168},
  year={2024}
}

@article{corvaja2017hilbert,
  title={On the {H}ilbert property and the fundamental group of algebraic varieties},
  author={Corvaja, Pietro and Zannier, Umberto},
  journal={Mathematische Zeitschrift},
  volume={286},
  number={1},
  pages={579--602},
  year={2017},
  publisher={Springer}
}

@article{corvaja2022distribution,
  title={On the distribution of rational points on ramified covers of abelian varieties},
  author={Corvaja, Pietro and Demeio, Julian Lawrence and Javanpeykar, Ariyan and Lombardo, Davide and Zannier, Umberto},
  journal={Compositio Mathematica},
  volume={158},
  number={11},
  pages={2109--2155},
  year={2022},
  publisher={London Mathematical Society}
}

@article{debarre1993abelian,
  title={Abelian varieties in ${W}^r_d({C})$ and points of bounded degree on algebraic curves},
  author={Debarre, Olivier and Fahlaoui, Rachid},
  journal={Compositio Mathematica},
  volume={88},
  number={3},
  pages={235--249},
  year={1993}
}

@article{frey1994curves,
  title={Curves with infinitely many points of fixed degree},
  author={Frey, Gerhard},
  journal={Israel Journal of Mathematics},
  volume={85},
  pages={79--83},
  year={1994},
  publisher={Springer}
}

@article{konig2024reducible,
  title={Reducible fibers of polynomial maps},
  author={K{\"o}nig, Joachim and Neftin, Danny},
  journal={International Mathematics Research Notices},
  volume={2024},
  number={6},
  pages={5373--5402},
  year={2024},
  publisher={Oxford University Press}
}

@article{dixon1996permutation,
  title={Permutation Groups},
  author={Dixon, John D and Mortimer, Brian},
  journal={Graduate Texts in Mathematics},
  volume={163},
  year={1996},
  publisher={Springer New York}
}

@article{neftin2024monodromy,
  title={Monodromy groups of indecomposable coverings of bounded genus},
  author={Neftin, Danny and Zieve, Michael E},
  journal={arXiv:2403.17167},
  year={2024}
}

@article{BKN,
  title={The {D}avenport--{L}ewis--{S}chinzel problem on the reducibility of $ f (X)-g (Y) $},
  author={Behajaina, Angelot and K{\"o}nig, Joachim and Neftin, Danny},
  journal={arXiv:2603.27728},
  year={2026}
}

@article{kadets2025subspace,
  title={Subspace configurations and low degree points on curves},
  author={Kadets, Borys and Vogt, Isabel},
  journal={Advances in Mathematics},
  volume={460},
  pages={110021},
  year={2025},
  publisher={Elsevier}
}

@article{hilbert1892ueber,
  title={{\"U}ber die {I}rreducibilit{\"a}t ganzer rationaler {F}unctionen mit ganzzahligen {C}oefficienten},
  author={Hilbert, David},
  journal={Journal f{\"u}r die reine und angewandte {M}athematik},
  volume={110},
  pages={104--129},
  year={1892}
}

@article{bochert1892ueber,
  title={{\"U}ber die {C}lasse der transitiven {S}ubstitutionengruppen},
  author={Bochert, Alfred},
  journal={Mathematische Annalen},
  volume={40},
  number={2},
  pages={176--193},
  year={1892},
  publisher={Springer-Verlag Berlin/Heidelberg}
}

@article{schomburg2022bochert,
  title={Bochert's results on the minimal degree of multiply transitive permutation groups},
  author={Schomburg, Bernd},
  journal={arXiv:2209.12049},
  year={2022}
}

@article{viray2024isolated,
  title={Isolated and parameterized points on curves},
  author={Viray, Bianca and Vogt, Isabel},
  journal={Essential Number Theory},
  volume={5},
  number={1},
  pages={1--47},
  year={2026},
  publisher={Mathematical Sciences Publishers}
}

@article {MN,
    AUTHOR = {Monderer, Tali and Neftin, Danny},
     TITLE = {Symmetric {G}alois groups under specialization},
   JOURNAL = {Israel J. Math.},
  FJOURNAL = {Israel Journal of Mathematics},
    VOLUME = {248},
      YEAR = {2022},
    NUMBER = {1},
     PAGES = {201--227},
      ISSN = {0021-2172,1565-8511},
   MRCLASS = {11R32 (11R09 12F12 14H30)},
  MRNUMBER = {4429280},
MRREVIEWER = {J\"urgen\ Kl\"uners},
       DOI = {10.1007/s11856-022-2302-x},
       URL = {https://doi-org.technion.idm.oclc.org/10.1007/s11856-022-2302-x},
}

@article{poonen2007gonality,
  title={Gonality of modular curves in characteristic p},
  author={Poonen, Bjorn},
  journal={Mathematical research letters},
  volume={14},
  number={4},
  pages={691--701},
  year={2007},
  publisher={International Press}
}

@misc{stacks-project,
    shorthand    = {Stacks},
    author       = {The {Stacks Project Authors}},
    title        = {\textit{Stacks Project}},
    howpublished = {\url{https://stacks.math.columbia.edu}},
    year         = {2018},
  }

@misc{fried86,
    author       = {Fried, Michael D.},
    title        = {Applications of the classification of finite simple groups to monodromy, part ii: {D}avenport and {H}ilbert--{S}iegel problems.},
    note = {preprint},
    year         = {1986},
  }

@article{bary2016hilbertian,
  title={Hilbertian fields and {G}alois representations.},
  author={Bary-Soroker, Lior and Fehm, Arno and Wiese, Gabor},
  journal={Journal f{\"u}r die Reine und Angewandte Mathematik},
  volume={2016},
  number={712},
  year={2016}
}

@article{gabber2013index,
  title={The index of an algebraic variety},
  author={Gabber, Ofer and Liu, Qing and Lorenzini, Dino},
  journal={Inventiones mathematicae},
  volume={192},
  number={3},
  pages={567--626},
  year={2013},
  publisher={Springer}
}

@article{fried2005field,
  title={Field Arithmetic},
  author={Fried, Michael D and Jarden, Moshe},
  journal={A Series of Modern Surveys in Mathematics (11)},
  year={2005},
  publisher={Berlin, Heidelberg: Springer-Verlag Berlin Heidelberg}
}

@article{debes1999integral,
  title={Integral specialization of families of rational functions},
  author={D{\`e}bes, Pierre and Fried, Michael D},
  journal={Pacific Journal of Mathematics},
  volume={190},
  number={1},
  pages={45--85},
  year={1999},
  publisher={Mathematical Sciences Publishers}
}

@incollection{deligne1981groupe,
  title={Le groupe fondamental du compl{\'e}ment d'une courbe plane n'ayant que des points doubles ordinaires est ab{\'e}lien [d'apr{\`e}s {W}. {F}ulton]},
  author={Deligne, Pierre},
  booktitle={S{\'e}minaire Bourbaki vol. 1979/80 Expos{\'e}s 543--560},
  pages={1--10},
  year={2006},
  publisher={Springer}
}

@book{serre1989lectures,
  title={Lectures on the {M}ordell--{W}eil theorem},
  author={Serre, Jean-Pierre and Brown, Martin and Waldschmidt, Michel},
  year={1989},
  publisher={Springer}
}

@book {ACGH,
    AUTHOR = {Arbarello, E. and Cornalba, M. and Griffiths, P. A. and
              Harris, J.},
     TITLE = {Geometry of algebraic curves. {V}ol. {I}},
    SERIES = {Grundlehren der mathematischen Wissenschaften [Fundamental
              Principles of Mathematical Sciences]},
    VOLUME = {267},
 PUBLISHER = {Springer-Verlag, New York},
      YEAR = {1985},
     PAGES = {xvi+386},
      ISBN = {0-387-90997-4},
   MRCLASS = {14Hxx (14-02)},
  MRNUMBER = {770932},
MRREVIEWER = {Werner Kleinert},
       DOI = {10.1007/978-1-4757-5323-3},
       URL = {https://doi.org/10.1007/978-1-4757-5323-3},
}

@article{van2026points,
  title={Points of low degree on curves over function fields},
  author={van Schaick, Si{\`e}na},
  journal={arXiv:2604.02975},
  year={2026}
}

@article{conrad2006chow,
  title={Chow's ${K}/k$-image and ${K}/k$-trace, and the {L}ang-{N}{\'e}ron theorem},
  author={Conrad, Brian},
  journal={Enseignement Math{\'e}matique},
  volume={52},
  number={1/2},
  pages={37},
  year={2006},
  publisher={SWETS \& ZEITLINGER}
}

@article{lang1959rational,
  title={Rational points of abelian varieties over function fields},
  author={Lang, Serge and N{\'e}ron, Andr{\'e}},
  journal={American Journal of Mathematics},
  volume={81},
  number={1},
  pages={95--118},
  year={1959},
  publisher={JSTOR}
}

@article{deligne1969irreducibility,
  title={The irreducibility of the space of curves of given genus},
  author={Deligne, Pierre and Mumford, David},
  journal={Publications Math{\'e}matiques de l'IHES},
  volume={36},
  pages={75--109},
  year={1969}
}

@article {Mul4,
    AUTHOR = {Müller, P.},
     TITLE = {Hilbert's irreducibility theorem for prime degree and general
              polynomials},
   JOURNAL = {Israel J. Math.},
  FJOURNAL = {Israel Journal of Mathematics},
    VOLUME = {109},
      YEAR = {1999},
     PAGES = {319--337},
      ISSN = {0021-2172,1565-8511},
   MRCLASS = {12E25},
MRREVIEWER = {N.\ Sankaran},
       DOI = {10.1007/BF02775041},
       URL = {https://doi.org/10.1007/BF02775041},
}

@incollection {Mul2,
    AUTHOR = {Müller, P.},
     TITLE = {Primitive monodromy groups of polynomials},
 BOOKTITLE = {Recent developments in the inverse {G}alois problem
              ({S}eattle, {WA}, 1993)},
    SERIES = {Contemp. Math.},
    VOLUME = {186},
     PAGES = {385--401},
 PUBLISHER = {Amer. Math. Soc., Providence, RI},
      YEAR = {1995},
      ISBN = {0-8218-0299-2},
   MRCLASS = {20B15 (11C08 12F10 20B20)},
MRREVIEWER = {Robert\ M.\ Guralnick},
       DOI = {10.1090/conm/186/02193},
       URL = {https://doi.org/10.1090/conm/186/02193},
}

@incollection {GN,
    AUTHOR = {Guralnick, Robert M. and Neubauer, Michael G.},
     TITLE = {Monodromy groups of branched coverings: the generic case},
 BOOKTITLE = {Recent developments in the inverse {G}alois problem
              ({S}eattle, {WA}, 1993)},
    SERIES = {Contemp. Math.},
    VOLUME = {186},
     PAGES = {325--352},
 PUBLISHER = {Amer. Math. Soc., Providence, RI},
      YEAR = {1995},
      ISBN = {0-8218-0299-2},
   MRCLASS = {20B25 (14H30 30F99)},
  MRNUMBER = {1352281},
MRREVIEWER = {Martin\ W.\ Liebeck},
       DOI = {10.1090/conm/186/02190},
       URL = {https://doi-org.technion.idm.oclc.org/10.1090/conm/186/02190},
}

@unpublished{DR25,
  author       = {Derickx, M. and Rawson, J.},
  title        = {Functions on curves with infinitely many split fibres},
  note         = {Work in preparation},
  year         = {2026},
}

@inproceedings{muller2002finiteness,
  title={Finiteness results for {H}ilbert's irreducibility theorem},
  author={M{\"u}ller, Peter},
  booktitle={Annales de l'institut {F}ourier},
  volume={52},
  number={4},
  pages={983--1015},
  year={2002}
}

@article{ellenberg2012expander,
  title={Expander graphs, gonality, and variation of Galois representations},
  author={Ellenberg, Jordan S and Hall, Chris and Kowalski, Emmanuel},
  journal={Duke Mathematical Journal},
  volume={161},
  number={7},
  pages={1233--1275},
  year={2012},
  publisher={Duke University Press}
}

@article{masser1989specializations,
  title={Specializations of finitely generated subgroups of abelian varieties},
  author={Masser, D.W.},
  journal={Transactions of the American Mathematical Society},
  volume={311},
  number={1},
  pages={413--424},
  year={1989}
}

@article{tucker2002irreducibility,
  title={Irreducibility, {B}rill--{N}oether loci, and {V}ojta’s inequality},
  author={Tucker, Thomas},
  journal={Transactions of the American Mathematical Society},
  volume={354},
  number={8},
  pages={3011--3029},
  year={2002},
  note={With an Appendix by Olivier Debarre}
}

@article {VdW,
    AUTHOR = {Bhargava, Manjul},
     TITLE = {Galois groups of random integer polynomials and van der
              {W}aerden's conjecture},
   JOURNAL = {Ann. of Math. (2)},
  FJOURNAL = {Annals of Mathematics. Second Series},
    VOLUME = {201},
      YEAR = {2025},
    NUMBER = {2},
     PAGES = {339--377},
      ISSN = {0003-486X,1939-8980},
   MRCLASS = {11R32 (11C08 11N35 11R45 20B15)},
  MRNUMBER = {4878222},
       DOI = {10.4007/annals.2025.201.2.1},
       URL = {https://doi.org/10.4007/annals.2025.201.2.1},
}

@book {Cameron,
    AUTHOR = {Cameron, Peter J.},
     TITLE = {Permutation groups},
    SERIES = {London Mathematical Society Student Texts},
    VOLUME = {45},
 PUBLISHER = {Cambridge University Press, Cambridge},
      YEAR = {1999},
     PAGES = {x+220},
      ISBN = {0-521-65302-9; 0-521-65378-9},
   MRCLASS = {20B40},
  MRNUMBER = {1721031},
MRREVIEWER = {Miguel\ \'A.\ Borges-Trenard},
       DOI = {10.1017/CBO9780511623677},
       URL = {https://doi.org/10.1017/CBO9780511623677},
}

@article{babai1981order,
  title={On the order of uniprimitive permutation groups},
  author={Babai, L{\'a}szl{\'o}},
  journal={Annals of Mathematics},
  volume={113},
  number={3},
  pages={553--568},
  year={1981},
  publisher={JSTOR}
}

@article {pyber,
    AUTHOR = {Pyber, L.},
     TITLE = {On the orders of doubly transitive permutation groups,
              elementary estimates},
   JOURNAL = {J. Combin. Theory Ser. A},
  FJOURNAL = {Journal of Combinatorial Theory. Series A},
    VOLUME = {62},
      YEAR = {1993},
    NUMBER = {2},
     PAGES = {361--366},
      ISSN = {0097-3165,1096-0899},
   MRCLASS = {20B20},
  MRNUMBER = {1207742},
MRREVIEWER = {Cheryl\ E.\ Praeger},
       DOI = {10.1016/0097-3165(93)90053-B},
       URL = {https://doi.org/10.1016/0097-3165(93)90053-B},
}

@article {GS,
    AUTHOR = {Guralnick, Robert M. and Shareshian, John},
     TITLE = {Symmetric and alternating groups as monodromy groups of
              {R}iemann surfaces. {I}. {G}eneric covers and covers with many
              branch points},
      NOTE = {With an appendix by Guralnick and R. Stafford},
   JOURNAL = {Mem. Amer. Math. Soc.},
  FJOURNAL = {Memoirs of the American Mathematical Society},
    VOLUME = {189},
      YEAR = {2007},
    NUMBER = {886},
     PAGES = {vi+128},
      ISSN = {0065-9266,1947-6221},
   MRCLASS = {14H30 (14H55 30F20)},
  MRNUMBER = {2343794},
MRREVIEWER = {Andrei\ B.\ Bogatyr\"ev},
       DOI = {10.1090/memo/0886},
       URL = {https://doi.org/10.1090/memo/0886},
}

@article {BG,
    AUTHOR = {Burness, Timothy C. and Guralnick, Robert M.},
     TITLE = {Fixed point ratios for finite primitive groups and
              applications},
   JOURNAL = {Adv. Math.},
  FJOURNAL = {Advances in Mathematics},
    VOLUME = {411},
      YEAR = {2022},
     PAGES = {Paper No. 108778, 90},
      ISSN = {0001-8708,1090-2082},
   MRCLASS = {20B15 (20D05 20P05)},
  MRNUMBER = {4512400},
MRREVIEWER = {Attila\ Mar\'oti},
       DOI = {10.1016/j.aim.2022.108778},
       URL = {https://doi.org/10.1016/j.aim.2022.108778},
}

@article {GT,
    AUTHOR = {Guralnick, Robert M. and Thompson, John G.},
     TITLE = {Finite groups of genus zero},
   JOURNAL = {J. Algebra},
  FJOURNAL = {Journal of Algebra},
    VOLUME = {131},
      YEAR = {1990},
    NUMBER = {1},
     PAGES = {303--341},
      ISSN = {0021-8693,1090-266X},
   MRCLASS = {20B25 (12F99 30F10)},
  MRNUMBER = {1055011},
MRREVIEWER = {Martin\ W.\ Liebeck},
       DOI = {10.1016/0021-8693(90)90178-Q},
       URL = {https://doi.org/10.1016/0021-8693(90)90178-Q},
}

@phdthesis{PHDTali,
author = {Monderer, T.},
title = {Reducibility, Specialization
and Related Low Genus Phenomena},
school = {Technion, Israel Institute of Technology},
year = {2024},
type = {Ph.{D}. thesis},
address = {Haifa, Israel},
}

@article{vojta1992generalization,
  title={A generalization of theorems of {F}altings and {T}hue-{S}iegel-{R}oth-{W}irsing},
  author={Vojta, Paul},
  journal={Journal of the American Mathematical Society},
  volume={5},
  number={4},
  pages={763--804},
  year={1992},
  publisher={JSTOR}
}

@article{song2001arithmetic,
  title={Arithmetic discriminants and morphisms of curves},
  author={Song, Xiangjun and Tucker, Thomas},
  journal={Transactions of the American Mathematical Society},
  volume={353},
  number={5},
  pages={1921--1936},
  year={2001}
}
\end{document}